\documentclass[a4paper,oneside,10pt]{article}%
\usepackage{amsmath}
\usepackage{amsfonts}
\usepackage{amssymb}
\usepackage{graphicx}
\usepackage{color}
\usepackage{bbm}
\usepackage[square,numbers,sort&compress]{natbib}%
\providecommand{\U}[1]{\protect \rule{.1in}{.1in}}
\pdfoutput=1

\newtheorem{theorem}{Theorem}[section]

\newtheorem{definition}[theorem]{Definition}
\newtheorem{assumption}[theorem]{Assumption}
\newtheorem{example}[theorem]{Example}

\newtheorem{lemma}[theorem]{Lemma}

\newtheorem{remark}[theorem]{Remark}

\newenvironment{proof}[1][Proof]{\noindent \textbf{#1.} }{\  \rule{0.5em}{0.5em}}
\numberwithin{equation}{section}
\begin{document}

\title{Error estimates for the robust $\alpha$-stable central limit theorem under
sublinear expectation by discrete approximation method}


\author{Lianzi Jiang\thanks{College of Mathematics and Systems Science, Shandong
University of Science and Technology, Qingdao, Shandong 266590, China.
jianglianzi95@163.com. Research supported by National Natural Science
Foundation of Shandong Province (No. ZR2023QA090)} }
\date{}
\maketitle

In this work, we develop a numerical method to study the error estimates of the $\alpha$-stable
central limit theorem under sublinear expectation with $\alpha \in(0,2)$, whose
limit distribution can be characterized by a fully nonlinear
integro-differential equation (PIDE). Based on the sequence of independent
random variables, we propose a discrete approximation scheme for the fully
nonlinear PIDE. With the help of the nonlinear stochastic analysis techniques
and numerical analysis tools, we establish the error bounds for the discrete
approximation scheme, which in turn provides a general error bound for the
robust $\alpha$-stable central limit theorem, including the integrable case
$\alpha \in(1,2)$ as well as the non-integrable case $\alpha \in(0,1]$. Finally,
we provide some concrete examples to illustrate our main results and derive
the precise convergence rates.


\section{Introduction}

The theory of robust probability and expectation has been developed by Peng
\cite{P2004,P2007,P20081,P2010}, who introduced the notion of sublinear
expectation space, called $G$-expectation space. This theory allows for the
evaluation of random outcomes over a family of possibly mutually singular
probability measures instead of a single probability measure. One of the
fundamental results in the theory is Peng's central limit theorem with laws of
large numbers established in \cite{P20082,P2010}. He showed that, under
certain moment conditions, the i.i.d. sequence $\{(X_{i},Y_{i})\}_{i=1}%
^{\infty}$ on a sublinear expectation space $(\Omega,\mathcal{H}%
,\mathbb{\hat{E}})$ converges in law to a $G$-distributed random variable
$(\xi,\eta)$ which can be characterized via a fully nonlinear parabolic PDE,
called $G$-equation
\[
\left \{
\begin{array}
[c]{l}%
\partial_{t}u(t,x,y)-G(D_{y}u,D_{x}^{2}u)=0,\\
u(0,x,y)=\phi(x,y),
\end{array}
\right.
\]
for any test function $\phi$. The discussion around convergence rates in
Peng's central limit theorem with laws of large numbers has been made by Song \cite{Song2020} and Fang et
al. \cite{FPSS2019} using Stein's method, and by Krylov \cite{Krylov2020} using
discrete stochastic control approximation under model uncertainty and further in
Huang and Liang \cite{HL2020} using monotone approximation scheme for
$G$-equation. For a recent account and development of $G$-expectation theory
and its applications, we refer the reader to
\cite{BMB2019,HTW2022,STZh2011,S2021} and the references therein.

As a generalization of $G$-Brownian motion, Hu and Peng \cite{HP2021}
introduced the concept of $G$-L\'{e}vy processes and established the
connection between $G$-L\'{e}vy processes and fully nonlinear
integro-differential equations (PIDEs) with finite jumps. Furthermore, the
case of infinite activity jumps has been studied in
\cite{DKN2020,Kuhn2019,NN2017,NR2021}. An important class of nonlinear
L\'{e}vy processes is the $\alpha$-stable process $(\zeta_{t})_{t\geq0}$ for
$\alpha \in(1,2)$, which corresponds to a fully nonlinear PIDE driven by a
family of $\alpha$-stable L\'{e}vy measures. On the basis of it, the related
$\alpha$-stable central limit theorem under sublinear expectation was
established by Bayraktar and Munk \cite{BM2016}. However, they require the
distribution assumption with respect to the solution of the fully nonlinear PIDE.

Based on the weak convergence approach, in collaboration with Hu, Liang, and Peng, the
author in \cite{HJLP2022} weakened the distribution assumption of the fully
nonlinear PIDE in \cite{BM2016} to that of test function, and established a
universal robust limit theorem, which encompasses Peng's robust central limit
theorem \cite{P20082,P2010} and Bayraktar-Munk's robust $\alpha$-stable limit
theorem \cite{BM2016} with $\alpha \in(1,2)$ as special cases. More recently,
together with Liang, the author in \cite{JL2023} relaxed the integrability condition in
\cite{HJLP2022} and established a new non-integrable $\alpha$-stable central
limit theorem for $\alpha \in(0,1]$. Specifically, together
with the case $\alpha \in(1,2)$, we show that the i.i.d. sequence of random
variables $\{Z_{i}\}_{i=1}^{\infty}$ on a sublinear expectation space
$(\Omega,\mathcal{H},\mathbb{\hat{E})}$ converges in law to a nonlinear
$\alpha$-stable distributed random variable $\tilde{\zeta}_{1}$ under moment
and consistency conditions, that is, for any $\phi \in C_{b,Lip}(\mathbb{R}%
^{d})$,
\begin{equation}
\lim_{n\rightarrow \infty}\mathbb{\hat{E}}\bigg[\phi \bigg(\frac{1}%
{\sqrt[\alpha]{n}}\sum_{i=1}^{n}Z_{i}\bigg)\bigg]=\mathbb{\tilde{E}}%
[\phi(\tilde{\zeta}_{1})].\label{0.0}%
\end{equation}
The limiting process\ $\tilde{\zeta}_{\cdot}$\ can be characterized by a fully
nonlinear PIDE
\begin{equation}
\left \{
\begin{array}
[c]{l}%
\displaystyle \partial_{t}u(t,x)-\sup \limits_{F_{\mu}\in \mathcal{L}}%
\int_{\mathbb{R}^{M}}\delta_{\lambda}^{\alpha}u(t,x)F_{\mu}(d\lambda)=0,\\
\displaystyle u(0,x)=\phi(x),\text{\  \  \ }\forall(t,x)\in \lbrack
0,1]\times \mathbb{R}^{d},
\end{array}
\right.  \label{0.1}%
\end{equation}
where $F_{\mu}$ is an $\alpha$-stable L\'{e}vy measure, $\mathcal{L}$ is a jump
uncertainty set, and the integral term is singular at the origin and given by
\[
\delta_{\lambda}^{\alpha}u(t,x):=\left \{
\begin{array}
[c]{ll}%
u(t,x+\lambda)-u(t,x)-\langle D_{x}u(t,x),\lambda \rangle, & \alpha \in(1,2),\\
u(t,x+\lambda)-u(t,x)-\langle D_{x}u(t,x),\lambda \mathbbm{1}_{\{|\lambda
|\leq1\}}\rangle, & \alpha=1,\\
u(t,x+\lambda)-u(t,x), & \alpha \in(0,1).
\end{array}
\right.
\]
When $\mathcal{L}$ is a singleton, $\tilde{\zeta}_{1}$ becomes an $\alpha
$-stable distributed random variable. The corresponding convergence rate of
the $\alpha$-stable central limit theorem has been studied by using the
characteristic function approach (see, e.g., \cite{DN2002,H19811,JP1998,KK2001}), and
by using Stein's method (see, e.g., \cite{CNXYZ2022,NP2012,X2019}). However,
the related tools in the sublinear expectation framework are still in infancy.
The goal of this work is to investigate the precise error bounds for the above
limit theorem by means of an analytical method.

The basic framework for convergence of numerical schemes to viscosity
solutions of first-order equations was initially established by Barles and
Souganidis \cite{BS1991}. By using a technique pioneered by Krylov based on
"shaking the coefficients" and mollification to construct smooth
subsolutions/supersolutions, the convergence rate for second-order equations
was first proved by Krylov in \cite{Krylov1997,Krylov1999,Krylov2000}. This
technique was further developed by Barles and Jakobsen on monotone
approximation schemes for local HJB equations (e.g.,
\cite{BJ2002,BJ2005,BJ2007,DJ2012}) and for nonlocal problems (e.g.,
\cite{BCJ2019,BJK2010,JKC2008}).

In this work, we shall give an error estimate result for the robust $\alpha$-stable
central limit theorem with $\alpha \in(0,2)$ by using discrete approximation schemes,
based on the observation that limiting distribution corresponds to the viscosity
solution of the fully nonlinear PIDE (\ref{0.1}). We first introduce a discrete approximation scheme using
the sequence $\{Z_{i}\}_{i=1}^{\infty}$. Then the error rates of the robust
$\alpha$-stable central limit theorem is transformed into the error bounds of
the discrete approximation scheme. It is well known that the error bounds rely
on the regularity of the discrete approximation scheme.
However, the random variable $Z_{1}$ has infinite variance. When
$\alpha \in(0,1]$, $Z_{1}$ is not even integrable, and the regularity estimates
method developed in \cite{Krylov2020} fails. To overcome it, we develop a new
$\delta$-estimate technique to obtain a general regularity result. Then by
using mollification method and nonlinear stochastic analysis approach, along
with the comparison principle and consistency estimates, we obtain the general
error bounds for our scheme, which in turn provides a general convergence
error for the robust $\alpha$-stable central limit theorem. Finally, we
provide some concrete examples to illustrate our main results and derive their
precise convergence rates. We also mention the work \cite{HJL2021}, where a
special sequence of random variables converging in law to an $\alpha$-stable
distributed with $\alpha \in(1,2)$ has been considered. To the best of our
knowledge, the present paper is the first dealing with convergence rate for
the robust $\alpha$-stable central limit theorem with $\alpha \in(0,2)$.

We organize this paper as follows. Section 2 gives some basic results on
sublinear expectation framework and assumptions on the robust $\alpha$-stable
central limit theorem. We shall propose a discrete approximation scheme
related to the limit theorem in Section 3 and study its properties, including
regularity, consistency, and comparison principle. Then we establish the
general error bounds for both the scheme and the limit theorem in Section 4.
In Section 5, some examples are given to derive the convergence rate in detail.

\section{Main assumptions and preliminaries}

\subsection{Sublinear expectation}

We first offer an extended sublinear expectation framework, which relaxes the
integrability requirement of random variables in Peng
\cite{P2007,P20081,P2010}. Let $\mathcal{H}$ be a linear space of real valued
functions on a given set $\Omega$ such that $\varphi(X_{1},\ldots,X_{n}%
)\in \mathcal{H}$ if $X_{1},\ldots,X_{n}\in \mathcal{H}$ for each $\varphi
\in{C_{b,Lip}}(\mathbb{R}^{n})$, the space of bounded and Lipschitz continuous
functions on $\mathbb{R}^{n}$. A sublinear expectation is a functional
$\mathbb{\hat{E}}:\mathcal{H}\rightarrow \mathbb{R}$ satisfying

\begin{description}
\item[(i)] (Monotonicity) $\mathbb{\hat{E}}[X]\geq \mathbb{\hat{E}}[Y]$, if
$X\geq Y$;

\item[(ii)] (Constant preservation) $\mathbb{\hat{E}}[c]=c$, for
$c\in \mathbb{R}$;

\item[(iii)] (Sub-additivity) $\mathbb{\hat{E}}[X+Y]\leq \mathbb{\hat{E}%
}[X]+\mathbb{\hat{E}}[Y]$;

\item[(iv)] (Positive homogeneity) $\mathbb{\hat{E}}[\lambda X]=\lambda
\mathbb{\hat{E}}[X]$, for $\lambda>0$.
\end{description}

The triplet $(\Omega,\mathcal{H},\mathbb{\hat{E})}$ is called a sublinear
expectation space. For an $n$-dimensional random variable $X$ defined on
$(\Omega,\mathcal{H},\mathbb{\hat{E})}$, we mean that $\varphi(X)\in
\mathcal{H}$ for all $\varphi \in C_{b,Lip}(\mathbb{R}^{n})$, and the random
variable $X$ itself is not required to be in $\mathcal{H}$.


\begin{definition}
\emph{(i)} Let $X_{1}$\ and $X_{2}$\ be two $n$-dimensional random variables
defined respectively on sublinear expectation spaces $(\Omega_{1}%
,\mathcal{H}_{1},\mathbb{\hat{E}}_{1})$ and $(\Omega_{2},\mathcal{H}%
_{2},\mathbb{\hat{E}}_{2})$. They are called identically distributed, denoted
by $X\overset{d}{=}Y$, if
\[
\mathbb{\hat{E}}_{1}[\varphi(X_{1})]=\mathbb{\hat{E}}_{2}[\varphi(X_{2})],
\]
for all $\varphi \in C_{b,Lip}(\mathbb{R}^{n})$.

\emph{(ii)} In a sublinear expectation space $(\Omega,\mathcal{H}%
,\mathbb{\hat{E})}$, a $m$-dimensional random variable $Y$, is said to be
independent from another $n$-dimensional random variable $X$ under
$\mathbb{\hat{E}}[\cdot]$, if for each $\varphi \in C_{b,Lip}(\mathbb{R}%
^{n+m})$
\[
\mathbb{\hat{E}}\left[  \varphi(X,Y)\right]  =\mathbb{\hat{E}}\left[
\mathbb{\hat{E}}\left[  \varphi(x,Y)\right]  _{x=X}\right]  ,
\]
which is denoted by $Y\perp \! \! \! \perp X$. If $Y\overset{d}{=}X$ and
$Y\perp \! \! \! \perp X$, we call $Y$ an independent copy of $X$.

\emph{(iii)} A sequence of $n$-dimensional random variables $\{X_{i}%
\}_{i=1}^{\infty}$ on a sublinear expectation space $(\Omega,\mathcal{H}%
,\mathbb{\hat{E}})$ is said to converge in law to $X$\ under $\mathbb{\hat{E}%
}$ if for each $\varphi \in C_{b,Lip}(\mathbb{R}^{n})$%
\[
\lim_{i\rightarrow \infty}\mathbb{\hat{E}}\left[  \varphi(X_{i})\right]
=\mathbb{\hat{E}}\left[  \varphi(X)\right]  .
\]

\end{definition}

\begin{definition}
Let $\alpha \in(0,2)$. An $n$-dimensional random variable $X$ is said to be
strictly $\alpha$-stable under a sublinear expectation space $(\Omega
,\mathcal{H},\mathbb{\hat{E})}$ if
\[
aX+bY\overset{d}{=}(a^{\alpha}+b^{\alpha})^{1/\alpha}X,\text{\ for }a,b\geq0,
\]
where $Y$ is an independent copy of $X$.
\end{definition}

\subsection{Robust $\alpha$-stable central limit theorem}

In this subsection, we review the robust $\alpha$-stable central limit theorem
developed in \cite{HJLP2022,JL2023} and introduce some properties of the
related fully nonlinear PIDE. We start by
collecting some useful notation which is needed frequently throughout this work.

Let $\alpha \in(0,2)$, $(\underline{\Lambda},\overline{\Lambda})$ for some
$\underline{\Lambda},\overline{\Lambda}>0$, and $F_{\mu}$ be the $\alpha
$-stable L\'{e}vy measure on $(\mathbb{R}^{M},\mathcal{B}(\mathbb{R}^{M}))$,
\begin{equation}
F_{\mu}(B)=\int_{S}\mu(dx)\int_{0}^{\infty}\mathbbm{1}_{B}(rx)\frac
{dr}{r^{1+\alpha}},\text{ \ for }B\in \mathcal{B}(\mathbb{R}^{M}), \label{F_mu}%
\end{equation}
where $\mu$ is a finite spectral measure on $S=\{x\in \mathbb{R}^{M}:|x|=1\}$. Introduce
a jump uncertainty set
\begin{equation}
\mathcal{L}=\left \{  F_{\mu}\  \text{measure on }\mathbb{R}^{M}:\mu
(S)\in(\underline{\Lambda},\overline{\Lambda})\right \}  . \label{L_0}%
\end{equation}

Let $\{Z_{i}\}_{i=1}^{\infty}$ be an i.i.d. sequence of $\mathbb{R}^{d}%
$-valued random variables defined on a sublinear expectation space
$(\Omega,\mathcal{H},\mathbb{\hat{E}})$, in the sense that $Z_{i+1}$
$\overset{d}{=}Z_{i}$ and $Z_{i+1}$ is independent from $(Z_{1},\ldots,Z_{i})$
for each $i\in \mathbb{N}$ satisfying the following assumptions:

\begin{assumption}
\label{assump1}

\begin{description}
\item[(i)] For $S_{n}:=\sum \limits_{i=1}^{n}Z_{i}$, $M_{\delta}:=\sup
\limits_{n}\mathbb{\hat{E}}[|n^{-\frac{1}{\alpha}}S_{n}|^{\delta}]<\infty$,
for some $0<\delta<\alpha$.

\item[(ii)] For each $\varphi \in C_{b}^{3}(\mathbb{R}^{d})$, the set of
functions with uniformly bounded derivatives up to the order $3$, satisfies
\[
\frac{1}{s}\bigg \vert \mathbb{\hat{E}}\big[\varphi(x+s^{\frac{1}{\alpha}%
}Z_{1})-\varphi(x)\big]-s\sup \limits_{F_{\mu}\in \mathcal{L}}\int
_{\mathbb{R}^{M}}\delta_{\lambda}^{\alpha}\varphi(x)F_{\mu}(d\lambda
)\bigg \vert \leq l_{\varphi}(s)\rightarrow0
\]
uniformly on $x\in \mathbb{R}^{d}$ as $s\rightarrow0$, where $l_{\varphi
}:[0,1]\rightarrow \mathbb{R}_{+}$ and
\[
\delta_{\lambda}^{\alpha}\varphi(x):=\left \{
\begin{array}
[c]{ll}%
\varphi(x+\lambda)-\varphi(x)-\langle D\varphi(x),\lambda \rangle, & \alpha
\in(1,2),\\
\varphi(x+\lambda)-\varphi(x)-\langle D\varphi(x),\lambda
\mathbbm{1}_{\{|\lambda|\leq1\}}\rangle, & \alpha=1,\\
\varphi(x+\lambda)-\varphi(x), & \alpha \in(0,1).
\end{array}
\right.
\]

\end{description}
\end{assumption}

\begin{remark}
From the moment condition (i) we can see that $\alpha=1$ is the critical
parameter for $Z_{i}$, $i\in \mathbb{N}$, since $Z_{i}$ is integrable if
$\alpha>1$, non-integrable otherwise. To be consistent with the assumption in \cite{HJLP2022},
in this paper, we always adopt $\delta=1$ in the case of $\alpha \in(1,2)$.
The condition (ii) ensures the consistency
condition when using the monotone approximation method (cf.
\cite{BJ2002,BJ2005,BJ2007}) to derive error estimates of the limit theorem.
For more details, we refer to Section \ref{Sec ex}, where the assumption above
will be illustrated with several examples.
\end{remark}

With the use of the weak convergence method and the L\'{e}vy-Khintchine
representation in \cite{HJLP2022,JL2023}, the corresponding robust $\alpha
$-stable central limit theorem and its PIDE characterization are established,
which gives the existence of the fully nonlinear PIDE (\ref{0.1}). It is also
standard to prove regularity results (see Theorem 4.9 in \cite{HJLP2022} and Theorem 4.7 in \cite{JL2023})
and uniqueness results (see Corollary 55 in \cite{HP2021} or Proposition 5.5 in \cite{NN2017}).

\begin{theorem}
\label{stable limit theorem}Suppose that Assumption \ref{assump1} holds. Then,
there exists a nonlinear $\alpha$-stable process $(\tilde{\zeta}_{t}%
)_{t\in \lbrack0,1]}$, connected with the jump uncertainty set $\mathcal{L}$
such that for any $\phi \in C_{b,Lip}(\mathbb{R}^{d})$,
\[
\lim_{n\rightarrow \infty}\mathbb{\hat{E}}\left[  \phi \left(  \frac{S_{n}%
}{\sqrt[\alpha]{n}}\right)  \right]  =\mathbb{\tilde{E}}[\phi(\tilde{\zeta
}_{1})]=u(1,0),
\]
where $u$ is the unique viscosity solution of the following fully nonlinear
PIDE
\begin{equation}
\left \{
\begin{array}
[c]{l}%
\displaystyle \partial_{t}u(t,x)-\sup \limits_{F_{\mu}\in \mathcal{L}}%
\int_{\mathbb{R}^{M}}\delta_{\lambda}^{\alpha}u(t,x)F_{\mu}(d\lambda)=0,\\
\displaystyle u(0,x)=\phi(x),\text{\  \  \ }\forall(t,x)\in \lbrack
0,1]\times \mathbb{R}^{d},
\end{array}
\right.  \label{PIDE}%
\end{equation}
where
\[
\delta_{\lambda}^{\alpha}u(t,x):=\left \{
\begin{array}
[c]{ll}%
u(t,x+\lambda)-u(t,x)-\langle D_{x}u(t,x),\lambda \rangle, & \alpha \in(1,2),\\
u(t,x+\lambda)-u(t,x)-\langle D_{x}u(t,x),\lambda \mathbbm{1}_{\{|\lambda
|\leq1\}}\rangle, & \alpha=1,\\
u(t,x+\lambda)-u(t,x), & \alpha \in(0,1).
\end{array}
\right.
\]
Furthermore, the following properties hold:

\begin{description}
\item[(i)] For $(t,x)\in \lbrack0,1]\times \mathbb{R}^{d}$,
$u(t,x)=\mathbb{\tilde{E}}[\phi(x+\tilde{\zeta}_{t})]$ is the unique viscosity
solution of (\ref{PIDE}) with a uniform bound $\left \Vert \phi \right \Vert
_{\infty}$, i.e., $\left \Vert u \right \Vert _{\infty}\leq\left \Vert \phi \right \Vert _{\infty}$.

\item[(ii)] For any $t,s\in \lbrack0,1]$ and $x,y\in \mathbb{R}^{d}$,
\begin{equation}
|u(t,x)-u(s,y)|\leq K(|t-s|^{\frac{\delta}{\alpha}}+|x-y|),
\label{u_regularity}%
\end{equation}
with $K:=C_{\phi}^{\delta}(2\left \Vert \phi \right \Vert _{\infty})^{1-\delta
}M_{\delta}\vee C_{\phi}$.

\item[(iii)] Let $u,-v$ be bounded upper semicontinuous functions on
$[0,1]\times \mathbb{R}^{d}$. If $u$ and $v$ are respectively viscosity sub-
and super- solutions of (\ref{PIDE}) and $u(0,\cdot)\leq v(0,\cdot)$, then
$u(t,\cdot)\leq v(t,\cdot)$ in $t\in(0,1].$
\end{description}
\end{theorem}

\section{Discrete approximation scheme}

Our main idea is to construct a discrete approximation scheme through the
random variables sequence $\{Z_{i}\}_{i=1}^{\infty}$, and prove the error
bounds between the numerical solution and the viscosity solution of
\eqref{PIDE}, which in turn establishes the error bounds of the robust
$\alpha$-stable central limit theorem. For any fixed $h\in(0,1)$, $\phi \in
C_{b,Lip}(\mathbb{R}^{d})$, and $Z\overset{d}{=}Z_{1}$, we define
$u_{h}:[0,1]\times \mathbb{R}^{d}\mathbb{\rightarrow R}$\ recursively by
\begin{equation}%
\begin{array}
[c]{l}%
u_{h}(t,x)=\phi(x)\  \  \  \text{in }[0,h)\times \mathbb{R}^{d},\\
u_{h}(t,x)=\mathbb{\hat{E}}\big[u_{h}(t-h,x+h^{\frac{1}{\alpha}}%
Z)\big]\  \  \text{\ in }[h,1]\times \mathbb{R}^{d}.
\end{array}
\label{2.2}%
\end{equation}
By induction, we can verify that for any $k\in \mathbb{N}$ such that $kh\leq1$
and $x\in \mathbb{R}^{d}$
\[
u_{h}(kh,x)=\mathbb{\hat{E}}\Big[\phi \Big(x+h^{\frac{1}{\alpha}}\sum
_{i=1}^{k}Z_{i}\Big)\Big].
\]
Specially, by taking $h=\frac{1}{n}$, we get%
\[
u_{h}(1,0)=\mathbb{\hat{E}}\left[  \phi \left(  \frac{S_{n}}{\sqrt[\alpha]{n}%
}\right)  \right]  .
\]

We first provide the space and time regularity properties of $u_{h}$, which
plays an important role in the subsequent error estimates.

\begin{theorem}
\label{uh_regularity}Suppose that Assumption \ref{assump1} (i) holds. Then

\begin{description}
\item[(i)] for any $t\in \lbrack0,1]$ and $x,y\in \mathbb{R}^{d}$
\begin{equation}
\left \vert u_{h}(t,x)-u_{h}(t,y)\right \vert \leq C_{\phi}|x-y|;
\label{uh_regularity 1}%
\end{equation}

\item[(ii)] for any $t,s\in \lbrack0,1]$ and $x\in \mathbb{R}^{d}$
\begin{equation}
\left \vert u_{h}(t,x)-u_{h}(s,x)\right \vert \leq C_{\phi}^{\delta}(2\left \Vert
\phi \right \Vert _{\infty})^{1-\delta}M_{\delta}(|t-s|^{\frac{\delta}{\alpha}%
}+h^{\frac{\delta}{\alpha}}), \label{uh_regularity 2}%
\end{equation}

\end{description}

where $C_{\phi}$ is the Lipschitz constant of $\phi$.
\end{theorem}

\begin{proof}
(i) The space regularity can be proved by induction with respect to $t$.
Indeed, (\ref{uh_regularity 1}) is obviously true for $t\in \lbrack0,h)$.
Suppose it holds true for $t\in \lbrack(k-1)h,kh)$ with $kh\leq1$. Then, for
$t\in \lbrack kh,(k+1)h\wedge1)$ and $x,y\in \mathbb{R}^{d}$
\[
\big \vert u_{h}(t,x)-u_{h}(t,y)\big \vert \leq \mathbb{\hat{E}}%
\big[\big \vert u_{h}(t-h,x+h^{\frac{1}{\alpha}}Z)-u_{h}(t-h,y+h^{\frac
{1}{\alpha}}Z)\big \vert \big]\leq C_{\phi}|x-y|.
\]

(ii) The time regularity will be proved in two steps. We first consider the
special case of $\left \vert u_{h}(kh,\cdot)-u_{h}(0,\cdot)\right \vert $ for
any $k\in \mathbb{N}$ with $kh\leq1$. By using (\ref{2.2}), we can recursively
obtain that
\[
u_{h}(kh,x)=\mathbb{\hat{E}}[\phi(x+h^{\frac{1}{\alpha}}S_{k})],
\]
for all $k\in \mathbb{N}$ with $kh\leq1$ and $x\in \mathbb{R}^{d}$. Note that
\begin{align*}
\left \vert \phi(x)-\phi(x^{\prime})\right \vert  &  \leq C_{\phi}\left(
|x^{\prime}-x|\wedge \frac{2\left \Vert \phi \right \Vert _{\infty}}{C_{\phi}%
}\right) \\
&  \leq C_{\phi}|x^{\prime}-x|^{\delta}\big(2\left \Vert \phi \right \Vert
_{\infty}C_{\phi}^{-1}\big)^{1-\delta}\\
&  =C_{\phi}^{\delta}(2\left \Vert \phi \right \Vert _{\infty})^{1-\delta
}|x^{\prime}-x|^{\delta},
\end{align*}
for all $x,x^{\prime}\in \mathbb{R}^{d}$ and some $0<\delta<\alpha$. Then,
under Assumption \ref{assump1}(i), we can deduce that \
\begin{align*}
|u_{h}(kh,x)-u_{h}(0,x)|  &  \leq \mathbb{\hat{E}}[|\phi(x+h^{\frac{1}{\alpha}%
}S_{k})-\phi(x)|]\\
&  \leq C_{\phi}^{\delta}(2\left \Vert \phi \right \Vert _{\infty})^{1-\delta
}\mathbb{\hat{E}}[|k^{-\frac{1}{\alpha}}S_{k}|^{\delta}](kh)^{\frac{\delta
}{\alpha}}\\
&  \leq C_{\phi}^{\delta}(2\left \Vert \phi \right \Vert _{\infty})^{1-\delta
}M_{\delta}(kh)^{\frac{\delta}{\alpha}},
\end{align*}
for all $k\in \mathbb{N}$ with $kh\leq1$ and $x\in \mathbb{R}^{d}$.\ This
follows that
\begin{equation}
|u_{h}(kh,x)-u_{h}(0,x)|\leq \tilde{C}(kh)^{\frac{\delta}{\alpha}}, \label{2.3}%
\end{equation}
where $\tilde{C}:=C_{\phi}^{\delta}(2\left \Vert \phi \right \Vert _{\infty
})^{1-\delta}M_{\delta}$.

We now consider the general case of $\left \vert u_{h}(t,\cdot)-u_{h}%
(s,\cdot)\right \vert $ for any $t,s\in \lbrack0,1]$. From (\ref{2.2}) and
(\ref{2.3}), we know that for any $x\in \mathbb{R}^{d}$ and $k,l\in \mathbb{N}$
such that $k\geq l$ and $(k\vee l)h\leq1$,%
\begin{align*}
\left \vert u_{h}(kh,x)-u_{h}(lh,x)\right \vert  &  =\big \vert \mathbb{\tilde
{E}}[u_{h}((k-l)h,x+h^{\frac{1}{\alpha}}S_{l})]-\mathbb{\tilde{E}}%
[u_{h}(0,x+h^{\frac{1}{\alpha}}S_{l})]\big \vert \\
&  \leq \mathbb{\tilde{E}}\big [\big \vert u_{h}((k-l)h,x+h^{\frac{1}{\alpha}%
}S_{l})-u_{h}(0,x+h^{\frac{1}{\alpha}}S_{l})\big \vert \big ]\\
&  \leq \tilde{C}((k-l)h)^{\frac{\delta}{\alpha}}.
\end{align*}
This yields that for $s,t\in \lbrack0,1]$, there exist constants $\delta
_{s},\delta_{t}\in \lbrack0,h)$ such that $s-\delta_{s}$ and $t-\delta_{t}$ are
in the grid points $\{kh:k\in \mathbb{N}\}$, the following inequality holds%
\begin{align*}
u_{h}(t,x)=u_{h}(t-\delta_{t},x)  &  \leq u_{h}(s-\delta_{s},x)+\tilde
{C}|t-s-\delta_{t}+\delta_{s}|^{\frac{\delta}{\alpha}}\\
&  \leq u_{h}(s,x)+\tilde{C}(|t-s|^{\frac{\delta}{\alpha}}+h^{\frac{\delta
}{\alpha}}).
\end{align*}
Similarly, we have%
\[
u_{h}(s,x)\leq u_{h}(t,x)+\tilde{C}(|t-s|^{\frac{\delta}{\alpha}}%
+h^{\frac{\delta}{\alpha}}).
\]
This implies the desired result.
\end{proof}

Let $\varepsilon \in(0,1)$ and $\zeta \in C^{\infty}(\mathbb{R}\times
\mathbb{R}^{d})$ be a nonnegative function with unit integral and support in
$[-1,0]\times B_{1}$. For any bounded function $v:[0,1]\times \mathbb{R}%
^{d}\rightarrow \mathbb{R}$\ with $\frac{1}{p}$-H\"{o}lder and Lipschitz
continuity in $(t,x)$, that is,
\[
|v(t,x)-v(s,y)|\leq C(|t-s|^{\frac{1}{p}}+|x-y|),\text{ \ for some }p>1,
\]
where $C>0$\ is a constant. Define the mollification of a suitable extension
of $v$ to $[0,1+\varepsilon^{p}]$
\begin{equation}
v^{\varepsilon}(t,x):=(v\ast \zeta_{\varepsilon,p})(t,x)=\int_{-\varepsilon
^{p}<\tau<0}\int_{|e|<\varepsilon}v(t-\tau,x-e)\zeta_{\varepsilon,p}%
(\tau,e)ded\tau, \label{v_varp}%
\end{equation}
where $\zeta_{\varepsilon,p}(t,x):=\varepsilon^{-(p+d)}\zeta(t/\varepsilon
^{p},x/\varepsilon)$ is an infinitely differentiable function. Clearly,
$v^{\varepsilon}\in C_{b}^{\infty}$ and the standard estimates for mollifiers
imply that
\begin{equation}
\left \Vert v-v^{\varepsilon}\right \Vert _{\infty}\leq2C\varepsilon \text{
\  \ and \  \ }\left \Vert \partial_{t}^{l}D_{x}^{k}v^{\varepsilon}\right \Vert
_{\infty}\leq2C\varepsilon^{1-pl-k}\text{ \  \ for\ }l,k\in \mathbb{N}\text{,}
\label{v_varp_estimate}%
\end{equation}
where $M_{\zeta}:=\max \limits_{k+l\geq1}\int_{-1<t<0}\int_{|x|<1}|\partial
_{t}^{l}D_{x}^{k}\zeta(t,x)|dxdt<\infty$.

Now we give the consistency properties of the discrete approximation scheme
(\ref{2.2}).

\begin{theorem}
\label{consistency}Suppose that Assumption \ref{assump1} holds. Let
$v^{\varepsilon}\in C_{b}^{\infty}\ $be the smooth function defined in
(\ref{v_varp}).\ Then the following inequality holds in $[h,1]\times
\mathbb{R}^{d}$%
\begin{align*}
&  \left \vert \partial_{t}v_{\varepsilon}(t,x)-\sup \limits_{F_{\mu}%
\in \mathcal{L}}\int_{\mathbb{R}^{M}}\delta_{\lambda}^{\alpha}v_{\varepsilon
}(t,x)F_{\mu}(d\lambda)-h^{-1}\left(  v_{\varepsilon}(t,x)-\mathbb{\hat{E}%
}\big[v_{\varepsilon}(t-h,x+h^{\frac{1}{\alpha}}Z)\big]\right)  \right \vert \\
&  \leq C\varepsilon^{1-2p}h+8C^{2}M_{\delta}\varepsilon^{1-p-\delta}%
h^{\frac{\delta}{\alpha}}+l_{v_{\varepsilon}}(h),
\end{align*}
for all $\varepsilon \in(0,1)$.
\end{theorem}

\begin{proof}
We first split the consistency error into two parts, for any $(t,x)\in \lbrack
h,1]\times \mathbb{R}^{d}$
\begin{align*}
R  &  :=\left \vert \partial_{t}v_{\varepsilon}(t,x)-\sup \limits_{F_{\mu}%
\in \mathcal{L}}\int_{\mathbb{R}^{M}}\delta_{\lambda}^{\alpha}v_{\varepsilon
}(t,x)F_{\mu}(d\lambda)-h^{-1}\left(  v_{\varepsilon}(t,x)-\mathbb{\hat{E}%
}\big[v_{\varepsilon}(t-h,x+h^{\frac{1}{\alpha}}Z)\big]\right)  \right \vert \\
&  \leq h^{-1}\left \vert \partial_{t}v_{\varepsilon}(t,x)h+\mathbb{\hat{E}%
}\big[v_{\varepsilon}(t-h,x+h^{\frac{1}{\alpha}}Z)\big]-\mathbb{\hat{E}%
}\big[v_{\varepsilon}(t,x+h^{\frac{1}{\alpha}}Z)\big]\right \vert \\
&  \  \  \ +h^{-1}\left \vert \mathbb{\hat{E}}\big[v_{\varepsilon}(t,x+h^{\frac
{1}{\alpha}}Z)\big]-v_{\varepsilon}(t,x)-h\sup \limits_{F_{\mu}\in \mathcal{L}%
}\int_{\mathbb{R}^{M}}\delta_{\lambda}^{\alpha}v_{\varepsilon}(t,x)F_{\mu
}(d\lambda)\right \vert \\
&  :=R_{1}+R_{2}.
\end{align*}
For the part $R_{1}$, from Taylor's expansion we have%
\[
v_{\varepsilon}(t,x+h^{\frac{1}{\alpha}}Z)=v_{\varepsilon}(t-h,x+h^{\frac
{1}{\alpha}}Z)+\int_{t-h}^{t}\partial_{t}v_{\varepsilon}(s,x+h^{\frac
{1}{\alpha}}Z)ds.
\]
This implies that%
\begin{align}
R_{1}  &  =h^{-1}\bigg \vert \mathbb{\hat{E}}\big[\partial_{t}v_{\varepsilon
}(t,x)h+v_{\varepsilon}(t-h,x+h^{\frac{1}{\alpha}}Z)\big]\nonumber \\
&  \text{ \  \ }-\mathbb{\hat{E}}\bigg [v_{\varepsilon}(t-h,x+h^{\frac
{1}{\alpha}}Z)+\int_{t-h}^{t}\partial_{t}v_{\varepsilon}(s,x+h^{\frac
{1}{\alpha}}Z)ds\bigg ]\bigg \vert \label{2.4}\\
&  \leq h^{-1}\int_{t-h}^{t}\mathbb{\hat{E}}\big[|\partial_{t}v_{\varepsilon
}(t,x)-\partial_{t}v_{\varepsilon}(s,x+h^{\frac{1}{\alpha}}%
Z)|\big]ds\nonumber \\
&  \leq \frac{1}{2}\left \Vert \partial_{t}^{2}v^{\varepsilon}\right \Vert
_{\infty}h+\left \Vert \partial_{t}D_{x}v^{\varepsilon}\right \Vert _{\infty
}^{\delta}(2\left \Vert \partial_{t}v^{\varepsilon}\right \Vert _{\infty
})^{1-\delta}h^{\frac{\delta}{\alpha}}M_{\delta},\nonumber
\end{align}
for all $(t,x)\in \lbrack h,1]\times \mathbb{R}^{d}$, where we have used the
fact that%
\begin{align*}
&  |\partial_{t}v_{\varepsilon}(t,x)-\partial_{t}v_{\varepsilon}%
(s,x+h^{\frac{1}{\alpha}}Z)|\\
&  \leq|\partial_{t}v_{\varepsilon}(t,x)-\partial_{t}v_{\varepsilon
}(s,x)|+|\partial_{t}v_{\varepsilon}(s,x)-\partial_{t}v_{\varepsilon
}(s,x+h^{\frac{1}{\alpha}}Z)|\\
&  \leq \left \Vert \partial_{t}^{2}v^{\varepsilon}\right \Vert _{\infty
}(t-s)+\left \Vert \partial_{t}D_{x}v^{\varepsilon}\right \Vert _{\infty
}^{\delta}(2\left \Vert \partial_{t}v^{\varepsilon}\right \Vert _{\infty
})^{1-\delta}h^{\frac{\delta}{\alpha}}\mathbb{\hat{E}}[|Z|^{\delta}].
\end{align*}
For the part $R_{2}$, applying Assumption \ref{assump1} (ii) to
$v^{\varepsilon}(t,\cdot)$, we see that $R_{2}\leq l_{v_{\varepsilon}}(h)$.
Combining this with (\ref{v_varp_estimate}) and (\ref{2.4}), from Assumption
\ref{assump1} (i), we have%
\[
R\leq C\varepsilon^{1-2p}h+8C^{2}M_{\delta}\varepsilon^{1-p-\delta}%
h^{\frac{\delta}{\alpha}}+l_{v_{\varepsilon}}(h),
\]
which we complete the proof.
\end{proof}

By means of the recursive structure of \eqref{2.2}, we obtain the following
comparison principle. For the proof of lemma, one can refer to Lemma 3.2 in
\cite{BJ2007}.

\begin{lemma}
\label{comparison}Suppose that $\underline{v},\bar{v}\in C_{b}([0,1]\times
\mathbb{R}^{d}\mathbb{)}$ and $h_{1},h_{2}\in C_{b}([h,1]\times \mathbb{R}%
^{d})$ satisfy
\begin{align*}
\frac{\underline{v}(t,x)-\mathbb{\hat{E}[}\underline{v}(t-h,x+h^{\frac
{1}{\alpha}}Z)]}{h}  &  \leq h_{1}(t,x)\quad \\
\frac{\bar{v}(t,x)-\mathbb{\hat{E}[}\bar{v}(t-h,x+h^{\frac{1}{\alpha}}Z)]}{h}
&  \geq h_{2}(t,x)
\end{align*}
for all $(t,x)\in \lbrack h,1]\times \mathbb{R}^{d}$. Then, for any
$(t,x)\in \lbrack0,1]\times \mathbb{R}^{d}$,
\[
\underline{v}-\bar{v}\leq \sup_{(t,x)\in \lbrack0,h)\times \mathbb{R}^{d}%
}(\underline{v}-\bar{v})^{+}+t\sup_{(t,x)\in \lbrack h,1]\times \mathbb{R}^{d}%
}(h_{1}-h_{2})^{+}\text{.}%
\]

\end{lemma}

\section{Error bounds}

In this section, we shall prove the error bounds of the discrete approximation
scheme $u_{\Delta}$. The convergence of the approximate solution $u_{\Delta}$
to the viscosity solution $u$ follows from a nonlocal extension of the
Barles-Souganidis half-relaxed limits method \cite{BS1991} and Krylov's
regularization results \cite{Krylov1997,Krylov2000} (see also Barles-Jakobsen
\cite{BJ2002,BJ2005}). Now we state our main result.

\begin{theorem}
\label{main theorem}Suppose that Assumption \ref{assump1} holds and
$h\in(0,1)$. Let $u$ be the unique solution of \eqref{PIDE} and let $u_{h}$ be
the unique solution of \eqref{2.2}. Then

\begin{description}
\item[(i)(Upper bound)] for all $\varepsilon \in(0,1)$%
\[
u_{h}-u\leq Kh^{^{\frac{\delta}{\alpha}}}+4K\varepsilon+\rho_{1}%
(\varepsilon,h)\text{ \  \ in }[0,1]\times \mathbb{R}^{d};
\]

\item[(ii)(Lower bound)] for all $\varepsilon \in(0,1)$%
\[
u-u_{h}\leq2Kh^{\frac{\delta}{\alpha}}+4K\varepsilon \text{ }+\rho
_{2}(\varepsilon,h)\text{ \  \ in }[0,1]\times \mathbb{R}^{d},
\]

\end{description}

where $\rho_{1}(\varepsilon,h)$ and $\rho_{2}(\varepsilon,h)$ are defined in
(\ref{rho1}) and (\ref{rho2}), respectively.
\end{theorem}

\begin{proof}
(i) Upper bound. In view of the continuity result in (\ref{u_regularity}), we
define $u^{\varepsilon}:=u\ast \zeta_{\varepsilon,p}$\ in (\ref{v_varp}) with
$p=\frac{\alpha}{\delta}$ satisfying%
\begin{equation}
\left \Vert u-u^{\varepsilon}\right \Vert _{\infty}\leq2K\varepsilon \text{
\  \ and \  \ }\left \Vert \partial_{t}^{l}D_{x}^{k}u^{\varepsilon}\right \Vert
_{\infty}\leq2K\varepsilon^{1-l\alpha/\delta-k}\text{ \  \ for\ }%
k,l\in \mathbb{N}. \label{u_mollifer}%
\end{equation}
It is worth noting that $u(t-\tau,x-e)$ is a viscosity solution of
(\ref{PIDE}) in $[0,1]\times \mathbb{R}^{d}$ for any $(\tau,e)\in
(-\varepsilon^{\alpha/\delta},0)\times B(0,\varepsilon)$. From the concavity
of (\ref{PIDE}) with respect to the nonlocal term, it follows that
$u^{\varepsilon}(t,x)$ is a supersolution of (\ref{PIDE}) in $(0,1]\times
\mathbb{R}^{d}$, that is, for $(t,x)\in(0,1]\times \mathbb{R}^{d}$,%
\begin{equation}
\partial_{t}u^{\varepsilon}(t,x)-\sup \limits_{F_{\mu}\in \mathcal{L}}%
\int_{\mathbb{R}^{M}}\delta_{\lambda}^{\alpha}u^{\varepsilon}(t,x)F_{\mu
}(d\lambda)\geq0. \label{3.1}%
\end{equation}
Theorem \ref{consistency}\ yields that%
\begin{equation}
\frac{u^{\varepsilon}(t,x)-\mathbb{\hat{E}}[u^{\varepsilon}(t-h,x+h^{\frac
{1}{\alpha}}Z)]}{h}\geq-K\varepsilon^{1-2\alpha/\delta}h-8K^{2}M_{\delta
}\varepsilon^{1-\alpha/\delta-\delta}h^{\delta/\alpha}-l_{u^{\varepsilon}%
}(h):=-\rho_{1}(\varepsilon,h). \label{rho1}%
\end{equation}
By using Lemma \ref{comparison} to compare $u_{h}$ and $u^{\varepsilon}$, we
have
\begin{equation}
u_{h}-u^{\varepsilon}\leq \sup_{(t,x)\in \lbrack0,h)\times \mathbb{R}^{d}}%
(u_{h}-u^{\varepsilon})^{+}+\rho_{1}(\varepsilon,h)\text{ \ in\ }%
[0,1]\times \mathbb{R}^{d}. \label{3.2}%
\end{equation}
In addition, from (\ref{u_regularity}) we see that
\begin{equation}
|u(t,x)-u_{h}(t,x)|=|u(t,x)-u(0,x)|\leq Kh^{\frac{\delta}{\alpha}} \label{3.3}%
\end{equation}
for all $(t,x)\in \lbrack0,h)\times \mathbb{R}^{d}$. Together with
(\ref{u_mollifer}) and (\ref{3.2})-(\ref{3.3}), we can deduce that
\begin{align*}
u_{h}-u  &  =u_{h}-u^{\varepsilon}+u^{\varepsilon}-u\\
&  \leq \sup_{(t,x)\in \lbrack0,h)\times \mathbb{R}^{d}}(u_{h}-u)^{+}+\left \Vert
u-u^{\varepsilon}\right \Vert _{\infty}+\rho_{1}(\varepsilon,h)+2K\varepsilon \\
&  \leq Kh^{\frac{\delta}{\alpha}}+4K\varepsilon+\rho_{1}(\varepsilon,h)\text{
\ in\ }[0,1]\times \mathbb{R}^{d}.
\end{align*}

(ii) Lower bound. For $\varepsilon \in(0,1)$, in view of Theorem
\ref{uh_regularity} (ii), we denote the mollification of $u_{h}$ by
$u_{h}^{\varepsilon}:=u_{h}\ast \zeta_{\varepsilon,p}$\ in (\ref{v_varp}) with
$p=\frac{\alpha}{\delta}$ satisfying%
\begin{equation}
\left \Vert u_{h}-u_{h}^{\varepsilon}\right \Vert _{\infty}\leq2K\varepsilon
\text{ \  \ and \  \ }\left \Vert \partial_{t}^{l}D_{x}^{k}u_{h}^{\varepsilon
}\right \Vert _{\infty}\leq2K(\varepsilon+h^{\frac{\delta}{\alpha}}%
)\varepsilon^{-l\alpha/\delta-k}\text{ \  \ for\ }k,l\in \mathbb{N}.
\label{uh_mollifer}%
\end{equation}
Using the convexity of $\mathbb{\hat{E}}$, we can derive that for each
$(t,x)\in \lbrack h,1]\times \mathbb{R}^{d}$
\[
\mathbb{\hat{E}}[u_{h}^{\varepsilon}(t-h,x+h^{\frac{1}{\alpha}}Z)]\leq
\mathbb{\hat{E}}[u_{h}(t-h,x+h^{\frac{1}{\alpha}}Z)]\ast \zeta_{\varepsilon
,p}=u_{h}(t,x)\ast \zeta_{\varepsilon,p}=u_{h}^{\varepsilon}(t,x),
\]
which implies that%
\[
\frac{u_{h}^{\varepsilon}(t,x)-\mathbb{\hat{E}}[u_{h}^{\varepsilon
}(t-h,x+h^{\frac{1}{\alpha}}Z)]}{h}\geq0.
\]
In view of Theorem \ref{consistency}, we then have for any $(t,x)\in \lbrack
h,1]\times \mathbb{R}^{d}$%
\begin{align}
&  \partial_{t}u_{h}^{\varepsilon}(t,x)-\sup \limits_{F_{\mu}\in \mathcal{L}%
}\int_{\mathbb{R}^{M}}\delta_{\lambda}^{\alpha}u_{h}^{\varepsilon}(t,x)F_{\mu
}(d\lambda)\label{rho2}\\
&  \geq-K(\varepsilon+h^{\frac{\delta}{\alpha}})\varepsilon^{-2\alpha/\delta
}h-8K^{2}M_{\delta}(\varepsilon+h^{\frac{\delta}{\alpha}})\varepsilon
^{-\alpha/\delta-\delta}h^{\delta/\alpha}-l_{u_{h}^{\varepsilon}}%
(h):=-\rho_{2}(\varepsilon,h).\nonumber
\end{align}
This means that
\[
\bar{u}(t,x):=u_{h}^{\varepsilon}(t,x)+\rho_{2}(\varepsilon
,h)(t-h)\text{\  \ is a viscosity supersolution of (\ref{PIDE}) in }%
(h,1]\times \mathbb{R}^{d}%
\]
with the initial condition $\bar{u}(h,x)=u_{h}^{\varepsilon}(h,x)$. In
addition, from (\ref{u_regularity}) and (\ref{uh_regularity 2}) we know that
\begin{equation}
|u-u_{h}|\leq2Kh^{\frac{\delta}{\alpha}}\text{ \ in }[0,h]\times \mathbb{R}%
^{d}. \label{3.4}%
\end{equation}
Combining this with the property of mollifier (\ref{uh_mollifer}), we get
\[
\underline{u}(t,x):=u(t,x)-2Kh^{\frac{\delta}{\alpha}}-2K\varepsilon
\text{\  \ is a viscosity subsolution of (\ref{PIDE}) in }(h,1]\times
\mathbb{R}^{d}%
\]
with the initial condition%
\begin{align*}
\underline{u}(h,x)  &  =u(h,x)-2Kh^{\frac{\delta}{\alpha}}-2K\varepsilon \\
&  =u_{h}^{\varepsilon}(h,x)+(u(h,x)-u_{h}(h,x))+(u_{h}(h,x)-u_{h}%
^{\varepsilon}(h,x))-2Kh^{\frac{\delta}{\alpha}}-2K\varepsilon \\
&  \leq u_{h}^{\varepsilon}(h,x)=\bar{u}(h,x).
\end{align*}
From the comparison principle (Theorem \ref{uh_regularity} (iii)), we have
$\underline{u}\leq \bar{u}$ in $[h,1]\times \mathbb{R}^{d}$. Then, it follows
from (\ref{uh_mollifer}) and (\ref{3.4}) that
\[
u-u_{h}\leq2Kh^{\frac{\delta}{\alpha}}+4K\varepsilon \text{ }+\rho
_{2}(\varepsilon,h)\text{\  \ in }[0,1]\times \mathbb{R}^{d},
\]
and this completes the proof.
\end{proof}

\begin{remark}
We remark that the error bounds in Theorem \ref{main theorem} always decrease
as $h$ decrease, but it is not applicable to $\varepsilon$. In fact, the
functions $\rho_{1}$ and $\rho_{2}$ explode as $\varepsilon \rightarrow0$, but
the error bounds can be small if and only if $\varepsilon$ is small.
Therefore, in order to obtain an optimal convergence rate, one has to minimize
the error bounds over $\varepsilon$ and the optimal choice will depend on $h$.
Some specific examples will be presented in Section \ref{Sec ex} to
demonstrate the above process and derive the optimal rate.
\end{remark}

The following result can be immediately obtain.

\begin{theorem}
\label{main theorem CLT}Suppose that Assumption \ref{assump1} holds. Then
\[
2Kn^{-\frac{\delta}{\alpha}}+4K\varepsilon \text{ }+\rho_{2}(\varepsilon
,n^{-1})\leq \mathbb{\hat{E}}\left[  \phi \left(  \frac{S_{n}}{\sqrt[\alpha]{n}%
}\right)  \right]  -\mathbb{\tilde{E}}[\phi(\tilde{\zeta}_{1})]\leq
Kn^{-\frac{\delta}{\alpha}}+4K\varepsilon+\rho_{1}(\varepsilon,n^{-1}),
\]
for all $n\in \mathbb{N}$, $\varepsilon \in(0,1)$, where $\rho_{1}%
(\varepsilon,h)$ and $\rho_{2}(\varepsilon,h)$ are given in Theorem
\ref{main theorem}.
\end{theorem}


\section{Examples\label{Sec ex}}

In this section, we shall apply our results to study the convergence rate of
several examples that have been considered in \cite{BM2016,HJL2021,HJLP2022,JL2023},
but with new or improved convergence rates. Let us start with some
notations and assumptions. Let $M=d=1$ and $F_{\mu}$ be the $\alpha$-stable
L\'{e}vy measure in (\ref{L_0}) with $\mu$ concentrated on $S=\{-1,1\}$, that
is,
\[
F_{\mu}(d\lambda)=\frac{k_{1}}{|\lambda|^{\alpha+1}}\mathbbm{1}_{(-\infty
,0)}(\lambda)d\lambda+\frac{k_{2}}{|\lambda|^{\alpha+1}}\mathbbm{1}_{(0,\infty
)}(\lambda)d\lambda,
\]
where $k_{1}=\mu \{-1\}$ and $k_{2}=\mu \{1\}$. Denote $k:=(k_{1},k_{2})$ and
$\Lambda:=(\underline{\Lambda},\overline{\Lambda})^{2}$. For each $k\in
\Lambda$, let $W_{k}$ be a classical random variable satisfying the cumulative
distribution function
\[
F_{W_{k}}(x)=\left \{
\begin{array}
[c]{ll}%
\displaystyle \left[  k_{1}/\alpha+\beta_{1,k}(x)\right]  \frac{1}%
{|x|^{\alpha}}, & x<0,\\
\displaystyle1-\left[  k_{2}/\alpha+\beta_{2,k}(x)\right]  \frac{1}{x^{\alpha
}}, & x>0,
\end{array}
\right.
\]
with some continuously differentiable functions $\beta_{1,k}:$ $(-\infty,0]$
$\rightarrow \mathbb{R}$ and $\beta_{2,k}:[0,\infty)\rightarrow \mathbb{R}$
satisfying
\[
\lim_{x\rightarrow-\infty}\beta_{1,k}(x)=\lim_{x\rightarrow \infty}\beta
_{2,k}(x)=0.
\]

More precisely, for $\alpha \in(1,2)$, we further assume that

\begin{description}
\item[(C1)] $W_{k}$ has mean zero.

\item[(C2)] there exist $q_{0}>0$ and $C_{\beta}>0$, such that the following
terms are less than $C_{\beta}n^{-q_{0}}$ for all $n\geq1$
\[%
\begin{array}
[c]{lll}%
\displaystyle|\beta_{1,k}(-n^{1/\alpha})|,\text{ \  \ } & \displaystyle \int
_{-\infty}^{-1}\frac{|\beta_{1,k}(n^{1/\alpha}x)|}{|x|^{\alpha}}dx,\text{ \ }
& \displaystyle \int_{-1}^{0}\frac{|\beta_{1,k}(n^{1/\alpha}x)|}%
{|x|^{\alpha-1}}dx,\\
&  & \\
\displaystyle|\beta_{2,k}(n^{1/\alpha})|,\text{ \ } & \displaystyle \int
_{1}^{\infty}\frac{|\beta_{2,k}(n^{1/\alpha}x)|}{x^{\alpha}}dx,\text{ \ } &
\displaystyle \int_{0}^{1}\frac{|\beta_{2,k}(n^{1/\alpha}x)|}{x^{\alpha-1}}dx.
\end{array}
\]

\end{description}

For $\alpha \in(0,1)$, we assume that

\begin{description}
\item[(D1)] $k_{1}=k_{2}$, and $\beta_{1,k}(x)=$ $\beta_{2,k}(-x)$ for any
$x\leq0$.

\item[(D2)] there exist $q_{0}>0$ and $C_{\beta}>0$, such that the following
terms are less than $C_{\beta}n^{-q_{0}}$ for all $n\geq1$%
\[%
\begin{array}
[c]{lll}%
\displaystyle|\beta_{2,k}(n^{1/\alpha})|,\text{ \ } & \displaystyle \int
_{1}^{\infty}\frac{|\beta_{2,k}(n^{1/\alpha}x)|}{x^{1+\alpha-\delta}}dx,\text{
\ } & \displaystyle \int_{0}^{1}\frac{|\beta_{2,k}(n^{1/\alpha}x)|}{x^{\alpha
}}dx.
\end{array}
\]

\end{description}

For $\alpha=1$, we assume that

\begin{description}
\item[(E1)] $k_{1}=k_{2}$, and $\beta_{1,k}(x)=$ $\beta_{2,k}(-x)$ for any
$x\leq0$.

\item[(E2)] there exist $q_{0}>0$ and $C_{\beta}>0$, such that the following
terms are less than $C_{\beta}n^{-q_{0}}$ for all $n\geq1$%
\[%
\begin{array}
[c]{lll}%
\displaystyle|\beta_{2,k}(n)|,\text{ \ } & \displaystyle \int_{1}^{\infty
}\frac{|\beta_{2,k}(nx)|}{x^{2-\delta}}dx,\text{ \ } & \displaystyle \int
_{0}^{1}|\beta_{2,k}(nx)|dx.
\end{array}
\]

\end{description}

Let $\Omega=\mathbb{R}$ and $\mathcal{H}_{0}$ be the space of bounded and
continuous functions on $\mathbb{R}$. For each $X=f(x)\in \mathcal{H}_{0}$,
define the sublinear expectation of $X$ by
\[
\mathbb{\hat{E}}[X]=\sup_{k\in \Lambda}\int_{\mathbb{R}}f(x)dF_{W_{k}}(x).
\]
We denote by $\mathcal{H}$\ the completion of $\mathcal{H}_{0}$ under the
norm$\  \left \Vert X\right \Vert :=\mathbb{\hat{E}}[|X|]$ and still denote the
extended sublinear expectation as$\  \mathbb{\hat{E}}$. Let
\[
Z(z)=z\text{, \ }z\in \mathbb{R},
\]
be a random variable, and $\{Z_{i}\}_{i=1}^{\infty}$ be a sequence of i.i.d.
$\mathbb{R}$-valued random variables in the sense that $Z_{1}\overset{d}{=}Z$,
$Z_{i+1}\overset{d}{=}Z_{i}$, and $Z_{i+1}\perp \! \! \! \perp(Z_{1}%
,Z_{2},\ldots,Z_{i})$ for every $i\in \mathbb{N}$.

We present a concrete example to illustrate the assumptions above.

\begin{example}
\label{example}Consider a sequence of random variables $\{Z_{i}\}_{i=1}%
^{\infty}$ defined on the sublinear expectation space $(\Omega,\mathcal{H}%
,\mathbb{\hat{E}})$ above with
\[
F_{W_{k}}(x)=\left \{
\begin{array}
[c]{ll}%
\displaystyle \left[  k_{1}/\alpha+a_{1}|x|^{\alpha-\beta}\right]  \frac
{1}{|x|^{\alpha}}, & x\leq-1,\\
\displaystyle1-\left[  k_{2}/\alpha+a_{2}x^{\alpha-\beta}\right]  \frac
{1}{x^{\alpha}}, & x\geq1,
\end{array}
\right.
\]
with $\beta>\alpha$ and some constants $a_{1},a_{2}>0$. Here we do not specify
$\beta_{1,k}(x)$ and $\beta_{2,k}(x)$ for $0<|x|<1$, but we require that
$\beta_{1,k}(x)$ and $\beta_{2,k}(x)$ satisfy Assumptions (C1), (D1), and (E1)
with respect to different $\alpha$. Then it can be checked that

1. When $\alpha \in(1,2)$, we have
\[
\int_{1}^{\infty}\frac{|\beta_{2,k}(n^{1/\alpha}x)|}{x^{\alpha}}dx=\frac
{a_{2}}{\beta-1}n^{-\frac{\beta-\alpha}{\alpha}}\leq O(n^{-q_{0}})
\]
and
\[
\int_{0}^{1}\frac{|\beta_{2,k}(n^{1/\alpha}x)|}{x^{\alpha-1}}dx\leq
b_{0}n^{-\frac{2-\alpha}{\alpha}}+a_{2}n^{\frac{\alpha-\beta}{\alpha}}%
\int_{n^{-1/\alpha}}^{1}x^{1-\beta}dx\leq O(n^{-q_{0}}),
\]
where $b_{0}:=\sup \limits_{x\in \lbrack0,1]}|\beta_{2,k}(n^{1/\alpha}x)|$, $\epsilon_{0}>0$ is sufficiently small, and
\[
q_{0}:=\left \{
\begin{array}
[c]{ll}%
\min \left \{  \frac{\beta-\alpha}{\alpha},\frac{2-\alpha}{\alpha}\right \}  , &
\beta \neq2,\beta>\alpha,\\
\frac{2-\alpha}{\alpha}-\epsilon_{0}, & \beta=2.
\end{array}
\right.
\]

2. When $\alpha \in(0,1)$, we have%
\[
\int_{1}^{\infty}\frac{|\beta_{2,k}(n^{1/\alpha}x)|}{x^{1+\alpha-\delta}%
}dx=\frac{a_{2}}{\beta-\delta}n^{-\frac{\beta-\alpha}{\alpha}}\leq
O(n^{-q_{0}})
\]
and
\[
\int_{0}^{1}\frac{|\beta_{2,k}(n^{1/\alpha}x)|}{x^{\alpha}}dx\leq
b_{0}n^{-\frac{1-\alpha}{\alpha}}+a_{2}n^{\frac{\alpha-\beta}{\alpha}}%
\int_{n^{-1/\alpha}}^{1}x^{-\beta}dx\leq O(n^{-q_{0}}),
\]
where $\epsilon_{0}>0$ is sufficiently small, and
\[
q_{0}:=\left \{
\begin{array}
[c]{ll}%
\min \left \{  \frac{\beta-\alpha}{\alpha},\frac{1-\alpha}{\alpha}\right \}  , &
\beta \neq1,\beta>\alpha,\\
\frac{1-\alpha}{\alpha}-\epsilon_{0}, & \beta=1.
\end{array}
\right.
\]

3. When $\alpha=1$, we have%
\[
\int_{1}^{\infty}\frac{|\beta_{2,k}(nx)|}{x^{2-\delta}}dx=\frac{a_{2}}%
{\beta-\delta}n^{-(\beta-1)}\leq O(n^{-q_{0}})
\]
and
\[
\int_{0}^{1}|\beta_{2,k}(nx)|dx\leq b_{0}n^{-1}+a_{2}n^{1-\beta}\int_{n^{-1}%
}^{1}x^{1-\beta}dx\leq O(n^{-q_{0}}),
\]
where $\epsilon_{0}>0$ is sufficiently small, and
\[
q_{0}:=\left \{
\begin{array}
[c]{ll}%
\min \left \{  \beta-1,1\right \}  , & \beta \neq2,\beta>1,\\
1-\epsilon_{0}, & \beta=2.
\end{array}
\right.
\]

\end{example}

From \cite{HJLP2022,JL2023}, we have known that the sequence $\{Z_{i}%
\}_{i=1}^{\infty}$ satisfies Assumption \ref{assump1}. In order to further
derive the optimal convergence rate, we will focus on the consistency
condition (Assumption \ref{assump1} (ii)) in the case of $\alpha \in(1,2)$,
$\alpha=1$, and $\alpha \in(0,1)$, respectively. We now state our convergence
rate result.

\begin{theorem}
\label{main theorem ex1 CLT}Suppose that the above assumptions hold. Then
\[
\left \vert \mathbb{\hat{E}}\left[  \phi \left(  \frac{S_{n}}{\sqrt[\alpha]{n}%
}\right)  \right]  -\mathbb{\tilde{E}}[\phi(\tilde{\zeta}_{1})]\right \vert
=Cn^{-\Gamma(\alpha,\delta,q_{0})}\text{ \ with }\Gamma(\alpha,\delta
,q_{0}):=\left \{
\begin{array}
[c]{ll}%
\min \left \{  \frac{2-\alpha}{2\alpha},\frac{q_{0}}{2}\right \}  , & \alpha
\in(1,2),\\
\min \left \{  \frac{\delta}{2},\frac{\delta^{2}}{1+\delta^{2}},\frac{q_{0}}%
{2}\right \}  , & \alpha=1,\\
\min \left \{  \frac{\delta}{2\alpha},\frac{\delta^{2}}{\alpha(\alpha+\delta
^{2})},\frac{1-\alpha}{\alpha},\alpha q_{0}\right \}  , & \alpha \in(0,1),
\end{array}
\right.
\]
where $\Gamma(\alpha,\delta,q_{0})$ is a positive constant depending
on\ $\alpha$, $\delta$, and $q_{0}$, and$\ C$ is a constant depending on $K$,
$M_{\delta}$, $\left \Vert \phi \right \Vert _{\infty}$, and $C_{\beta}$.
\end{theorem}

\begin{remark}
In the case $\alpha \in(1,2)$, our results have improved the convergence rate
$\min \left \{  \frac{1}{4},\frac{2-\alpha}{2\alpha},\frac{q_{0}}{2}\right \} $
in \cite{HJL2021}.
\end{remark}

\begin{remark}
Applying Theorem \ref{main theorem ex1 CLT} to Example \ref{example}, we have
when $\alpha \in(1,2)$,%
\[
\Gamma(\alpha,\delta,\beta)=\left \{
\begin{array}
[c]{ll}%
\min \left \{  \frac{\beta-\alpha}{2\alpha},\frac{2-\alpha}{2\alpha}\right \}
, & \beta \neq2,\beta>\alpha,\\
\frac{2-\alpha}{2\alpha}-\epsilon_{0}, & \beta=2,
\end{array}
\right.
\]
when $\alpha \in(0,1)$, we have
\[
\Gamma(\alpha,\delta,\beta)=\left \{
\begin{array}
[c]{ll}%
\min \left \{  \frac{\delta}{2\alpha},\frac{\delta^{2}}{\alpha(\alpha+\delta
^{2})},\beta-\alpha,1-\alpha \right \}  , & \beta \neq1,\beta>\alpha,\\
\min \left \{  \frac{\delta}{2\alpha},\frac{\delta^{2}}{\alpha(\alpha+\delta
^{2})},1-\alpha-\epsilon_{0}\right \}  , & \beta=1,
\end{array}
\right.
\]
and when $\alpha=1$, we have%
\[
\Gamma(\alpha,\delta,\beta)=\left \{
\begin{array}
[c]{ll}%
\min \left \{  \frac{\delta}{2},\frac{\delta^{2}}{1+\delta^{2}},\frac{\beta
-1}{2}\right \}  , & \beta \neq2,\beta>1,\\
\min \left \{  \frac{\delta}{2},\frac{\delta^{2}}{1+\delta^{2}},\frac{1}%
{2}-\epsilon_{0}\right \}  , & \beta=2,
\end{array}
\right.
\]
with any small $\epsilon_{0}>0$.
\end{remark}

\subsection{The case of $\alpha \in(1,2)$}

For $\varphi \in C_{b}^{3}(\mathbb{R})$ and $(s,x)\in(0,1]\times \mathbb{R}$,
using a change of variables and the assumption (C1), we can deduce that
\begin{align*}
&  \frac{1}{s}\bigg \vert \mathbb{\hat{E}}\big[\varphi(x+s^{\frac{1}{\alpha}%
}Z_{1})-\varphi(x)\big]-s\sup \limits_{F_{\mu}\in \mathcal{L}}\int_{\mathbb{R}%
}\delta_{\lambda}^{\alpha}\varphi(x)F_{\mu}(d\lambda)\bigg \vert \\
&  \leq \sup_{k\in \Lambda}\bigg \vert \int_{\mathbb{-\infty}}^{0}%
\delta_{\lambda}^{\alpha}\varphi(x)[\alpha \beta_{1,k}(s^{-\frac{1}{\alpha}%
}\lambda)-\beta_{1,k}^{\prime}(s^{-\frac{1}{\alpha}}\lambda)s^{-\frac
{1}{\alpha}}\lambda]|\lambda|^{-\alpha-1}d\lambda \\
&  +\int_{0}^{\infty}\delta_{\lambda}^{\alpha}\varphi(x)[\alpha \beta
_{2,k}(s^{-\frac{1}{\alpha}}\lambda)-\beta_{2,k}^{\prime}(s^{-\frac{1}{\alpha
}}\lambda)s^{-\frac{1}{\alpha}}\lambda]\lambda^{-\alpha-1}d\lambda \bigg \vert,
\end{align*}
for all $x\in \mathbb{R}$. In the following, we consider the integral above
along the positive half-line, and the\ integral along the negative half-line
can be similarly obtained. For any given $k\in \Lambda$, we denote%
\begin{align*}
&  \left \vert \int_{0}^{\infty}\delta_{\lambda}^{\alpha}\varphi(x)[\alpha
\beta_{2,k}(s^{-\frac{1}{\alpha}}\lambda)-\beta_{2,k}^{\prime}(s^{-\frac
{1}{\alpha}}\lambda)s^{-\frac{1}{\alpha}}\lambda]\lambda^{-\alpha-1}%
d\lambda \right \vert \\
&  \leq \left \vert \int_{1}^{\infty}\delta_{\lambda}^{\alpha}\varphi
(x)[\alpha \beta_{2,k}(s^{-\frac{1}{\alpha}}\lambda)-\beta_{2,k}^{\prime
}(s^{-\frac{1}{\alpha}}\lambda)s^{-\frac{1}{\alpha}}\lambda]\lambda
^{-\alpha-1}d\lambda \right \vert \\
&  \text{\  \  \ }+\left \vert \int_{s^{\frac{1}{\alpha}}}^{1}\delta_{\lambda
}^{\alpha}\varphi(x)[\alpha \beta_{2,k}(s^{-\frac{1}{\alpha}}\lambda
)-\beta_{2,k}^{\prime}(s^{-\frac{1}{\alpha}}\lambda)s^{-\frac{1}{\alpha}%
}\lambda]\lambda^{-\alpha-1}d\lambda \right \vert \\
&  \text{ \  \ }+\bigg \vert \int_{0}^{s^{\frac{1}{\alpha}}}\delta_{\lambda
}^{\alpha}\varphi(x)[\alpha \beta_{2,k}(s^{-\frac{1}{\alpha}}\lambda
)-\beta_{2,k}^{\prime}(s^{-\frac{1}{\alpha}}\lambda)s^{-\frac{1}{\alpha}%
}\lambda]\lambda^{-\alpha-1}d\lambda \bigg \vert \\
&  :=I_{1}+I_{2}+I_{3}.
\end{align*}
Note that $\frac{\partial}{\partial \lambda}\delta_{\lambda}^{\alpha}%
\varphi(x)=D\varphi(x+\lambda)-D\varphi(x)$ and
\[
\frac{\partial}{\partial \lambda}\big(-\beta_{2,k}(s^{-\frac{1}{\alpha}}%
\lambda)\lambda^{-\alpha}\big)=\big(\alpha \beta_{2,k}(s^{-\frac{1}{\alpha}%
}\lambda)-\beta_{2,k}^{\prime}(s^{-\frac{1}{\alpha}}\lambda)s^{-\frac
{1}{\alpha}}\lambda \big)\lambda^{-\alpha-1}.
\]
Using integration by parts, we derive that
\begin{align*}
I_{1}  &  =\bigg \vert \delta_{1}^{\alpha}\varphi(x)\beta_{2,k}(s^{-\frac
{1}{\alpha}})+\int_{1}^{\infty}\beta_{2,k}(s^{-\frac{1}{\alpha}}%
\lambda)(D\varphi(x+\lambda)-D\varphi(x))\lambda^{-\alpha}d\lambda
\bigg \vert \\
&  \leq2\left \Vert D\varphi \right \Vert _{\infty}\bigg (|\beta_{2,k}%
(s^{-\frac{1}{\alpha}})|+\int_{1}^{\infty}|\beta_{2,k}(s^{-\frac{1}{\alpha}%
}\lambda)|\lambda^{-\alpha}d\lambda \bigg),
\end{align*}
where we have used the fact that
\[
\delta_{1}^{\alpha}\varphi(x)=\int_{0}^{1}(D\varphi(x+\theta)-D\varphi
(x))d\theta \leq2\left \Vert D\varphi \right \Vert _{\infty}.
\]
Also, by means of integration by parts, it follows that
\begin{align*}
I_{2}  &  =\bigg \vert \delta_{s^{1/\alpha}}^{\alpha}\varphi(x)\beta
_{2,k}(1)s^{-1}-\delta_{1}^{\alpha}\varphi(x)\beta_{2,k}(s^{-\frac{1}{\alpha}%
})\\
\text{ }  &  \  \  \ +\int_{s^{\frac{1}{\alpha}}}^{1}\beta_{2,k}(s^{-\frac
{1}{\alpha}}\lambda)(D\varphi(x+\lambda)-D\varphi(x))\lambda^{-\alpha}%
d\lambda \bigg \vert \\
&  \leq \frac{1}{2}\left \Vert D^{2}\varphi \right \Vert _{\infty}|\beta
_{2,k}(1)|s^{\frac{2-\alpha}{\alpha}}+2\left \Vert D\varphi \right \Vert
_{\infty}|\beta_{2,k}(s^{-\frac{1}{\alpha}})|\\
&  \text{ \  \ }+\left \Vert D^{2}\varphi \right \Vert _{\infty}\int_{0}^{1}%
|\beta_{2,k}(s^{-\frac{1}{\alpha}}\lambda)|\lambda^{1-\alpha}d\lambda
\end{align*}
where we have used the fact that
\[
\delta_{s^{1/\alpha}}^{\alpha}\varphi(x)=\int_{0}^{1}\int_{0}^{1}D^{2}%
\varphi(z+\tau \theta s^{\frac{1}{\alpha}})s^{\frac{2}{\alpha}}\theta d\tau
d\theta \leq \frac{1}{2}\left \Vert D^{2}\varphi \right \Vert _{\infty}s^{\frac
{2}{\alpha}}.
\]
\qquad By changing variables, we obtain that
\begin{align*}
I_{3}  &  \leq \left \Vert D^{2}\varphi \right \Vert _{\infty}\int_{0}%
^{s^{\frac{1}{\alpha}}}\big \vert \alpha \beta_{2,k}(s^{-\frac{1}{\alpha}%
}\lambda)-\beta_{2,k}^{\prime}(s^{-\frac{1}{\alpha}}\lambda)s^{-\frac
{1}{\alpha}}\lambda \big \vert \lambda^{1-\alpha}d\lambda \\
&  =\left \Vert D^{2}\varphi \right \Vert _{\infty}s^{\frac{2-\alpha}{\alpha}%
}\int_{0}^{1}\big \vert \alpha \beta_{2,k}(\lambda)-\beta_{2,k}^{\prime
}(\lambda)\lambda \big \vert \lambda^{1-\alpha}d\lambda.
\end{align*}
Since for any $k\in \Lambda$%
\[
\beta_{2,k}(x)=(1-F_{W_{k}}(x))x^{\alpha}-\frac{k_{2}}{\alpha},\text{\  \ }%
x\geq0,
\]
it is straightforward to check that the following terms are uniformly bounded
(also denote as $C_{\beta}$)
\[%
\begin{array}
[c]{lll}%
\displaystyle|\beta_{2,k}(1)|, &  & \displaystyle \int_{0}^{1}\frac
{|\alpha \beta_{2,k}(\lambda)-\beta_{2,k}^{\prime}(\lambda)\lambda|}%
{\lambda^{\alpha-1}}d\lambda.
\end{array}
\]
Thus, we have%
\begin{align*}
&  \frac{1}{s}\bigg \vert \mathbb{\hat{E}}\big[\varphi(x+s^{\frac{1}{\alpha}%
}Z_{1})-\varphi(x)\big]-s\sup \limits_{F_{\mu}\in \mathcal{L}}\int_{\mathbb{R}%
}\delta_{\lambda}^{\alpha}\varphi(x)F_{\mu}(d\lambda)\bigg \vert \\
&  \leq4\left \Vert D\varphi \right \Vert _{\infty}\sup_{k\in \Lambda
}\bigg \{|\beta_{1,k}(-s^{-\frac{1}{\alpha}})|+|\beta_{2,k}(s^{-\frac
{1}{\alpha}})|\\
&  \text{ \  \ }+\int_{1}^{\infty}\left[  |\beta_{1,k}(-s^{-\frac{1}{\alpha}%
}\lambda)|+|\beta_{2,k}(s^{-\frac{1}{\alpha}}\lambda)|\right]  \lambda
^{-\alpha}d\lambda \bigg \} \\
&  \text{ \  \ }+\left \Vert D^{2}\varphi \right \Vert _{\infty}\sup_{k\in \Lambda
}\int_{0}^{1}\left[  |\beta_{1,k}(-s^{-\frac{1}{\alpha}}\lambda)|+|\beta
_{2,k}(s^{-\frac{1}{\alpha}}\lambda)|\right]  \lambda^{1-\alpha}d\lambda \\
&  \text{ \  \ }+\left \Vert D^{2}\varphi \right \Vert _{\infty}s^{\frac{2-\alpha
}{\alpha}}\sup_{k\in \Lambda}\bigg \{|\beta_{1,k}(-1)|+|\beta_{2,k}(1)|\\
&  \text{ \  \ }+\int_{0}^{1}\left[  |\alpha \beta_{1,k}(-\lambda)+\beta
_{1,k}^{\prime}(-\lambda)\lambda|+|\alpha \beta_{2,k}(\lambda)-\beta
_{2,k}^{\prime}(\lambda)\lambda|\right]  \lambda^{1-\alpha}d\lambda \bigg \} \\
&  \leq(16C_{\beta}\left \Vert D\varphi \right \Vert _{\infty}+2C_{\beta
}\left \Vert D^{2}\varphi \right \Vert _{\infty})s^{q_{0}}+4C_{\beta}\left \Vert
D^{2}\varphi \right \Vert _{\infty}s^{\frac{2-\alpha}{\alpha}},
\end{align*}
for all $s\in(0,1]$ and $x\in \mathbb{R}$. Using the condition (C2), we
conclude that%
\begin{equation}
l_{\varphi}(s)=(16C_{\beta}\left \Vert D\varphi \right \Vert _{\infty}+2C_{\beta
}\left \Vert D^{2}\varphi \right \Vert _{\infty})s^{q_{0}}+4C_{\beta}\left \Vert
D^{2}\varphi \right \Vert _{\infty}s^{\frac{2-\alpha}{\alpha}}. \label{ex1}%
\end{equation}

\subsection{The case of $\alpha \in(0,1)$}

For $\varphi \in C_{b}^{3}(\mathbb{R})$ and $(s,x)\in(0,1]\times \mathbb{R}$, by
a change of variables, we can also derive that
\begin{align*}
&  \frac{1}{s}\bigg \vert \mathbb{\hat{E}}\big[\varphi(x+s^{\frac{1}{\alpha}%
}Z_{1})-\varphi(x)\big]-s\sup \limits_{F_{\mu}\in \mathcal{L}}\int_{\mathbb{R}%
}\delta_{\lambda}^{\alpha}\varphi(x)F_{\mu}(d\lambda)\bigg \vert \\
&  \leq \sup_{k\in \Lambda}\bigg \vert \int_{\mathbb{-\infty}}^{0}%
\delta_{\lambda}^{\alpha}\varphi(x)\left[  \alpha \beta_{2,k}(-s^{-\frac
{1}{\alpha}}\lambda)-\beta_{2,k}^{\prime}(-s^{-\frac{1}{\alpha}}%
\lambda)s^{-\frac{1}{\alpha}}\lambda \right]  |\lambda|^{-\alpha-1}d\lambda \\
&  \text{ \  \ }+\int_{0}^{\infty}\delta_{\lambda}^{\alpha}\varphi(x)\left[
\alpha \beta_{2,k}(s^{-\frac{1}{\alpha}}\lambda)-\beta_{2,k}^{\prime}%
(s^{-\frac{1}{\alpha}}\lambda)s^{-\frac{1}{\alpha}}\lambda \right]
\lambda^{-\alpha-1}d\lambda \bigg \vert.
\end{align*}
For each $k\in \Lambda$, we set
\begin{align*}
&  \bigg \vert \int_{0}^{\infty}\delta_{\lambda}^{\alpha}\varphi(x)\left[
\alpha \beta_{2,k}(s^{-\frac{1}{\alpha}}\lambda)-\beta_{2,k}^{\prime}%
(s^{-\frac{1}{\alpha}}\lambda)s^{-\frac{1}{\alpha}}\lambda \right]
\lambda^{-\alpha-1}d\lambda \bigg \vert \\
&  =\bigg \vert \int_{1}^{\infty}\delta_{\lambda}^{\alpha}\varphi(x)\left[
\alpha \beta_{2,k}(s^{-\frac{1}{\alpha}}\lambda)-\beta_{2,k}^{\prime}%
(s^{-\frac{1}{\alpha}}\lambda)s^{-\frac{1}{\alpha}}\lambda \right]
\lambda^{-\alpha-1}d\lambda \bigg \vert \\
&  \text{ \  \ }+\bigg \vert \int_{s^{\frac{1}{\alpha}}}^{1}\delta_{\lambda
}^{\alpha}\varphi(x)\left[  \alpha \beta_{2,k}(s^{-\frac{1}{\alpha}}%
\lambda)-\beta_{2,k}^{\prime}(s^{-\frac{1}{\alpha}}\lambda)s^{-\frac{1}%
{\alpha}}\lambda \right]  \lambda^{-\alpha-1}d\lambda \bigg \vert \\
&  \text{ \  \ }+\bigg \vert \int_{0}^{s^{\frac{1}{\alpha}}}\delta_{\lambda
}^{\alpha}\varphi(x)\left[  \alpha \beta_{2,k}(s^{-\frac{1}{\alpha}}%
\lambda)-\beta_{2,k}^{\prime}(s^{-\frac{1}{\alpha}}\lambda)s^{-\frac{1}%
{\alpha}}\lambda \right]  \lambda^{-\alpha-1}d\lambda \bigg \vert \\
&  :=J_{1}+J_{2}+J_{3}\text{.}%
\end{align*}
For the part $J_{1}$, when $q_{0}\in(\frac{1-\alpha}{\alpha},\infty)$, in the
same way we can get%
\begin{align*}
J_{1}  &  =\bigg \vert \delta_{1}^{\alpha}\varphi(x)\beta_{2,k}(s^{-\frac
{1}{\alpha}})+\int_{1}^{\infty}D\varphi(x+\lambda)\beta_{2,k}(s^{-\frac
{1}{\alpha}}\lambda)\lambda^{-\alpha}d\lambda \bigg \vert \\
&  \leq2\left \Vert \varphi \right \Vert _{\infty}|\beta_{2,k}(s^{-\frac
{1}{\alpha}})|+C_{\beta}\left \Vert D\varphi \right \Vert _{\infty}\int
_{1}^{\infty}\frac{d\lambda}{\lambda^{\alpha(1+q_{0})}}s^{q_{0}}\\
&  \leq2\left \Vert \varphi \right \Vert _{\infty}|\beta_{2,k}(s^{-\frac
{1}{\alpha}})|+\frac{C_{\beta}}{\alpha(1+q_{0})-1}\left \Vert D\varphi
\right \Vert _{\infty}s^{q_{0}}\text{.}%
\end{align*}
When $q_{0}\in \lbrack0,\frac{1-\alpha}{\alpha}]$, from (D2), we choose some
$N_{0}>1$ such that $|\beta_{2,k}(x)|\leq C_{\beta}$, for $|x|\geq N_{0}$.
This leads to
\begin{align*}
&  \bigg \vert \int_{N_{0}}^{\infty}\delta_{\lambda}^{\alpha}\varphi
(x)\big[\alpha \beta_{2,k}(s^{-\frac{1}{\alpha}}\lambda)-\beta_{2,k}^{\prime
}(s^{-\frac{1}{\alpha}}\lambda)s^{-\frac{1}{\alpha}}\lambda \big]\lambda
^{-\alpha-1}d\lambda \bigg \vert \\
&  \leq2\left \Vert \varphi \right \Vert _{\infty}\int_{N_{0}}^{\infty
}\big \vert \alpha \beta_{2,k}(s^{-\frac{1}{\alpha}}\lambda)-\beta
_{2,k}^{\prime}(s^{-\frac{1}{\alpha}}\lambda)s^{-\frac{1}{\alpha}}%
\lambda \big \vert \lambda^{-\alpha-1}d\lambda \\
&  =2\left \Vert \varphi \right \Vert _{\infty}|\beta_{2,k}(s^{-\frac{1}{\alpha}%
}N_{0})|N_{0}^{-\alpha}\leq2C_{\beta}\left \Vert \varphi \right \Vert _{\infty
}N_{0}^{-\alpha},
\end{align*}
and
\begin{align*}
&  \bigg \vert \int_{1}^{N_{0}}\delta_{\lambda}^{\alpha}\varphi(x)\big[\alpha
\beta_{2,k}(s^{-\frac{1}{\alpha}}\lambda)-\beta_{2,k}^{\prime}(s^{-\frac
{1}{\alpha}}\lambda)s^{-\frac{1}{\alpha}}\lambda \big]\lambda^{-\alpha
-1}d\lambda \bigg \vert \\
&  =\bigg \vert \delta_{1}^{\alpha}\varphi(x)\beta_{2,k}(s^{-\frac{1}{\alpha}%
})-\delta_{N_{0}}^{\alpha}\varphi(x)\beta_{2,k}(s^{-\frac{1}{\alpha}}%
N_{0})N_{0}^{-\alpha}+\int_{1}^{N_{0}}D\varphi(x+\lambda)\beta_{2,k}%
(s^{-\frac{1}{\alpha}}\lambda)\lambda^{-\alpha}d\lambda \bigg \vert \\
&  \leq \left \Vert D\varphi \right \Vert _{\infty}|\beta_{2,k}(s^{-\frac
{1}{\alpha}})|+2C_{\beta}\left \Vert \varphi \right \Vert _{\infty}N_{0}%
^{-\alpha(1+q_{0})}+\frac{1}{1-\alpha}\left \Vert D\varphi \right \Vert _{\infty
}\sup_{|\lambda|\geq s^{-\frac{1}{\alpha}}}|\beta_{2,k}(\lambda)|N_{0}%
^{1-\alpha}\\
&  \leq \left \Vert D\varphi \right \Vert _{\infty}|\beta_{2,k}(s^{-\frac
{1}{\alpha}})|+2C_{\beta}\left \Vert \varphi \right \Vert _{\infty}N_{0}%
^{-\alpha}+\frac{C_{\beta}}{1-\alpha}\left \Vert D\varphi \right \Vert _{\infty
}s^{q_{0}}N_{0}^{1-\alpha}%
\end{align*}
By letting $N_{0}=s^{-q_{0}}$, we obtain that%
\[
J_{1}\leq \left \Vert D\varphi \right \Vert _{\infty}|\beta_{2,k}(s^{-\frac
{1}{\alpha}})|+C_{\beta}\left(  4\left \Vert \varphi \right \Vert _{\infty}%
+\frac{1}{1-\alpha}\left \Vert D\varphi \right \Vert _{\infty}\right)  s^{\alpha
q_{0}}\text{.}%
\]
For the part $J_{2}$, using integration by parts, we can deduce that
\begin{align*}
J_{2}  &  =\bigg \vert \delta_{s^{\frac{1}{\alpha}}}^{\alpha}\varphi
(x)\beta_{2,k}(1)s^{-1}-\delta_{1}^{\alpha}\varphi(x)\beta_{2,k}(s^{-\frac
{1}{\alpha}})+\int_{s^{\frac{1}{\alpha}}}^{1}D\varphi(x+\lambda)\beta
_{2,k}(s^{-\frac{1}{\alpha}}\lambda)\lambda^{-\alpha}d\lambda \bigg \vert \\
&  \leq \left \Vert D\varphi \right \Vert _{\infty}\left(  |\beta_{2,k}%
(1)|s^{\frac{1-\alpha}{\alpha}}+|\beta_{2,k}(s^{-\frac{1}{\alpha}})|+\int
_{0}^{1}|\beta_{2,k}(s^{-\frac{1}{\alpha}}\lambda)|\lambda^{-\alpha}%
d\lambda \right)  \text{.}%
\end{align*}
By changing variables, we get that
\begin{align*}
J_{3}  &  =\bigg \vert \int_{0}^{1}\delta_{s^{\frac{1}{\alpha}}\lambda
}^{\alpha}\varphi(x)\left[  \alpha \beta_{2,k}(\lambda)-\beta_{2,k}^{\prime
}(\lambda)\lambda \right]  s^{-1}\lambda^{-\alpha-1}d\lambda \bigg \vert \\
&  \leq \left \Vert D\varphi \right \Vert _{\infty}s^{\frac{1-\alpha}{\alpha}}%
\int_{0}^{1}\big \vert \alpha \beta_{2,k}(\lambda)-\beta_{2,k}^{\prime}%
(\lambda)\lambda \big \vert \lambda^{-\alpha}d\lambda.
\end{align*}
Note that the following terms are uniformly bounded (also denote as $C_{\beta
}$)
\[%
\begin{array}
[c]{lll}%
\displaystyle|\beta_{2,k}(1)|, &  & \displaystyle \int_{0}^{1}\frac
{|\alpha \beta_{2,k}(\lambda)-\beta_{2,k}^{\prime}(\lambda)\lambda|}%
{\lambda^{\alpha}}d\lambda.
\end{array}
\]
This indicates that%
\begin{align*}
&  \frac{1}{s}\bigg \vert \mathbb{\hat{E}}\big[\varphi(x+s^{\frac{1}{\alpha}%
}Z_{1})-\varphi(x)\big]-s\sup \limits_{F_{\mu}\in \mathcal{L}}\int_{\mathbb{R}%
}\delta_{\lambda}^{\alpha}\varphi(x)F_{\mu}(d\lambda)\bigg \vert \\
&  \leq4\left \Vert D\varphi \right \Vert _{\infty}\sup_{k\in \Lambda
}\bigg \{|\beta_{2,k}(s^{-\frac{1}{\alpha}})|+\int_{0}^{1}|\beta
_{2,k}(s^{-\frac{1}{\alpha}}\lambda)|\lambda^{-\alpha}d\lambda \bigg \} \\
&  \text{ \  \ }+2\left \Vert D\varphi \right \Vert _{\infty}s^{\frac{1-\alpha
}{\alpha}}\sup_{k\in \Lambda}\bigg \{|\beta_{2,k}(1)|+\int_{0}^{1}|\alpha
\beta_{2,k}(\lambda)-\beta_{2,k}^{\prime}(\lambda)\lambda|\lambda^{-\alpha
}d\lambda \bigg \} \\
&  \text{ \  \ }+C_{\beta}\left(  8\left \Vert \varphi \right \Vert _{\infty
}+\frac{2}{1-\alpha}\left \Vert D\varphi \right \Vert _{\infty}\right)  s^{\alpha
q_{0}},
\end{align*}
for all $(s,x)\in(0,1]\times \mathbb{R}$. From (D2), we conclude that%
\begin{equation}
l_{\varphi}(s)=8C_{\beta}\left \Vert D\varphi \right \Vert _{\infty}s^{q_{0}%
}+4C_{\beta}\left \Vert D\varphi \right \Vert _{\infty}s^{\frac{1-\alpha}{\alpha
}}+C_{\beta}\left(  8\left \Vert \varphi \right \Vert _{\infty}+\frac{2}%
{1-\alpha}\left \Vert D\varphi \right \Vert _{\infty}\right)  s^{\alpha q_{0}}.
\label{ex2}%
\end{equation}

\subsection{The case of $\alpha=1$}

For $\varphi \in C_{b}^{3}(\mathbb{R})$ and $(s,x)\in(0,1]\times \mathbb{R}$,
using a change of variables and the assumption (E1), it follows that
\begin{align*}
&  \frac{1}{s}\bigg \vert \mathbb{\hat{E}}\big[\varphi(x+sZ_{1})-\varphi
(x)\big]-s\sup \limits_{F_{\mu}\in \mathcal{L}}\int_{\mathbb{R}}\delta_{\lambda
}^{1}\varphi(x)F_{\mu}(d\lambda)\bigg \vert \\
&  \leq \sup_{k\in \Lambda}\bigg \vert \int_{\mathbb{-\infty}}^{0}%
\delta_{\lambda}^{1}\varphi(x)[\beta_{1,k}(s^{-1}\lambda)-\beta_{1,k}^{\prime
}(s^{-1}\lambda)s^{-1}\lambda]|\lambda|^{-2}d\lambda \\
&  \text{ \  \ }+\int_{0}^{\infty}\delta_{\lambda}^{1}\varphi(x)[\beta
_{2,k}(s^{-1}\lambda)-\beta_{2,k}^{\prime}(s^{-1}\lambda)s^{-1}\lambda
]\lambda^{-2}d\lambda \bigg \vert.
\end{align*}
For each $k\in \Lambda$, we set
\begin{align*}
&  \bigg \vert \int_{0}^{\infty}\delta_{\lambda}^{1}\varphi(x)[\alpha
\beta_{2,k}(s^{-1}\lambda)-\beta_{2,k}^{\prime}(s^{-1}\lambda)s^{-1}%
\lambda]\lambda^{-2}d\lambda \bigg \vert \\
&  =\bigg \vert \int_{1}^{\infty}\delta_{\lambda}^{1}\varphi(x)\big[\alpha
\beta_{2,k}(s^{-1}\lambda)-\beta_{2,k}^{\prime}(s^{-1}\lambda)s^{-1}%
\lambda \big]\lambda^{-2}d\lambda \bigg \vert \\
&  \text{\  \  \ }+\bigg \vert \int_{s}^{1}\delta_{\lambda}^{1}\varphi
(x)\big[\alpha \beta_{2,k}(s^{-1}\lambda)-\beta_{2,k}^{\prime}(s^{-1}%
\lambda)s^{-1}\lambda \big]\lambda^{-2}d\lambda \bigg \vert \\
&  \text{\  \  \ }+\bigg \vert \int_{0}^{s}\delta_{\lambda}^{1}\varphi
(x)\big[\alpha \beta_{2,k}(s^{-1}\lambda)-\beta_{2,k}^{\prime}(s^{-1}%
\lambda)s^{-1}\lambda \big]\lambda^{-2}d\lambda \bigg \vert \\
&  :=T_{1}+T_{2}+T_{3}\text{.}%
\end{align*}
Similar to the previous cases, using integration by parts and changing
variables, we can check that
\begin{align*}
T_{1}  &  =\bigg \vert(\varphi(x+1)-\varphi(x))\beta_{2,k}(s^{-1})+\int
_{1}^{\infty}D\varphi(x+\lambda)\beta_{2,k}(s^{-1}\lambda)\lambda^{-1}%
d\lambda \bigg \vert \\
&  \leq \left \Vert D\varphi \right \Vert _{\infty}|\beta_{2,k}(s^{-1}%
)|+\left \Vert D\varphi \right \Vert _{\infty}\int_{1}^{\infty}|\beta
_{2,k}(s^{-1}\lambda)|\lambda^{-1}d\lambda \text{,}%
\end{align*}%
\begin{align*}
T_{2}  &  =\bigg \vert \delta_{s}^{1}\varphi(x)\beta_{2,k}(1)s^{-1}-\delta
_{1}^{1}\varphi(x)\beta_{2,k}(s^{-1})\\
&  \text{ \  \ }+\int_{s}^{1}(D\varphi(x+\lambda)-D\varphi(x))\beta
_{2,k}(s^{-1}\lambda)\lambda^{-1}d\lambda \bigg \vert \\
&  \leq \left \Vert D^{2}\varphi \right \Vert _{\infty}(|\beta_{2,k}%
(1)|s+|\beta_{2,k}(s^{-1})|)+\left \Vert D^{2}\varphi \right \Vert _{\infty}%
\int_{0}^{1}|\beta_{2,k}(s^{-1}\lambda)|d\lambda \text{,}%
\end{align*}
and%
\begin{align*}
T_{3}  &  =\bigg \vert s^{-1}\int_{0}^{1}(\varphi(x+s\lambda)-\varphi
(x)-D\varphi(x)s\lambda)(\beta_{2,k}(\lambda)-\lambda \beta_{2,k}^{\prime
}(\lambda))\lambda^{-2}d\lambda \bigg \vert \\
&  \leq \frac{s}{2}\left \Vert D^{2}\varphi \right \Vert _{\infty}\int_{0}%
^{1}|\beta_{2,k}(\lambda)-\lambda \beta_{2,k}^{\prime}(\lambda)|d\lambda
\text{.}%
\end{align*}
This follows immediately that, for all $(s,x)\in(0,1]\times \mathbb{R}$
\begin{align*}
&  \frac{1}{s}\bigg \vert \mathbb{\hat{E}}\big[\varphi(x+sZ_{1})-\varphi
(x)\big]-s\sup \limits_{F_{\mu}\in \mathcal{L}}\int_{\mathbb{R}}\delta_{\lambda
}^{1}\varphi(x)F_{\mu}(d\lambda)\bigg \vert \\
&  \leq2\left \Vert D^{2}\varphi \right \Vert _{\infty}\sup_{k\in \Lambda}\int
_{0}^{1}|\beta_{2,k}(s^{-1}\lambda)|d\lambda+2\left \Vert D\varphi \right \Vert
_{\infty}\int_{1}^{\infty}|\beta_{2,k}(s^{-1}\lambda)|\lambda^{-1}d\lambda \\
&  \text{ \  \ }+s\left \Vert D^{2}\varphi \right \Vert _{\infty}\sup_{k\in
\Lambda}\int_{0}^{1}|\beta_{2,k}(\lambda)-\lambda \beta_{2,k}^{\prime}%
(\lambda)|d\lambda \\
&  \text{ \  \ }+2(\left \Vert D\varphi \right \Vert _{\infty}+\left \Vert
D^{2}\varphi \right \Vert _{\infty})\sup_{k\in \Lambda}|\beta_{2,k}%
(s^{-1})|+2s\left \Vert D^{2}\varphi \right \Vert _{\infty}\sup_{k\in \Lambda
}|\beta_{2,k}(1)|.
\end{align*}
Together with (E2) and the uniformly bounded (also denote as $C_{\beta}$) of
\[%
\begin{array}
[c]{lll}%
\displaystyle|\beta_{2,k}(1)| & \text{ and } & \displaystyle \int_{0}%
^{1}|\beta_{2,k}(\lambda)-\lambda \beta_{2,k}^{\prime}(\lambda)|d\lambda,
\end{array}
\]
we can conclude that
\begin{equation}
l_{\varphi}(s)=3C_{\beta}\left \Vert D^{2}\varphi \right \Vert _{\infty}s+\left(
4\left \Vert D^{2}\varphi \right \Vert _{\infty}+2(1+q_{0}^{-1})\left \Vert
D\varphi \right \Vert _{\infty}\right)  C_{\beta}s^{q_{0}}. \label{ex3}%
\end{equation}

\subsection{Proof of Theorem \ref{main theorem ex1 CLT}}

Based on the above discussion and Theorem \ref{main theorem}, we can derive
the following convergence rate for the discrete approximation scheme
(\ref{2.2}). Then, Theorem \ref{main theorem ex1 CLT} is a direct consequence of
this result.

\begin{theorem}
\label{main theorem ex1}Suppose that the assumptions hold. Then
\[
\left \vert u-u_{h}\right \vert \leq Ch^{\Gamma(\alpha,\delta,q_{0})}\text{ in
}[0,1]\times \mathbb{R}\text{,\  \ with }\Gamma(\alpha,\delta,q_{0})=\left \{
\begin{array}
[c]{ll}%
\min \left \{  \frac{2-\alpha}{2\alpha},\frac{q_{0}}{2}\right \}  , & \alpha
\in(1,2),\\
\min \left \{  \frac{\delta}{2},\frac{\delta^{2}}{1+\delta^{2}},\frac{q_{0}}%
{2}\right \}  , & \alpha=1,\\
\min \left \{  \frac{\delta}{2\alpha},\frac{\delta^{2}}{\alpha(\alpha+\delta
^{2})},\frac{1-\alpha}{\alpha},\alpha q_{0}\right \}  , & \alpha \in(0,1),
\end{array}
\right.
\]
where $C$ is a constant depending on $K$, $M_{\delta}$, $\left \Vert
\phi \right \Vert _{\infty}$, and $C_{\beta}$.
\end{theorem}

\begin{proof}
(i) For $\alpha \in(1,2)$. From Theorem \ref{main theorem} (i),
(\ref{u_mollifer}), and (\ref{ex1}), we know that%
\begin{align*}
u_{h}-u  &  \leq Kh^{1/\alpha}+4K\varepsilon+K\varepsilon^{1-2\alpha}%
h+8K^{2}M_{\delta}\varepsilon^{-\alpha}h^{1/\alpha}+l_{u^{\varepsilon}}(h)\\
&  \leq Kh^{1/\alpha}+4K\varepsilon+K\varepsilon^{1-2\alpha}h+8K^{2}M_{\delta
}\varepsilon^{-\alpha}h^{1/\alpha}\\
&  \text{ \  \ }+C_{\beta}K\left(  32+4\varepsilon^{-1}\right)  h^{q_{0}%
}+8C_{\beta}K\varepsilon^{-1}h^{\frac{2-\alpha}{\alpha}}.
\end{align*}
Let $\varepsilon=h^{\gamma}$ with $\gamma>0$. Then, we have%
\begin{align*}
u_{h}-u  &  \leq Kh^{1/\alpha}+4Kh^{\gamma}+Kh^{(1-2\alpha)\gamma+1}%
+8K^{2}M_{\delta}h^{-\alpha \gamma+1/\alpha}\\
&  \text{ \  \ }+36C_{\beta}Kh^{q_{0}-\gamma}+8C_{\beta}Kh^{\frac{2-\alpha
}{\alpha}-\gamma}\text{ \ in }[0,1]\times \mathbb{R}.
\end{align*}
In order to minimize the right-hand side, we need to maximize $\gamma$ under
the constraints $\gamma \leq(1-2\alpha)\gamma+1$, $\gamma \leq-\alpha
\gamma+1/\alpha$, $\gamma \leq q_{0}-\gamma$, $\gamma \leq \frac{2-\alpha}%
{\alpha}-\gamma$,\ and\ $\gamma \leq1/\alpha$. Hence, by choosing
\begin{equation}
\varepsilon=h^{\gamma}\text{\ with \ }\gamma:=\min \left \{  \frac{2-\alpha
}{2\alpha},\frac{q_{0}}{2}\right \}  , \label{Gama}%
\end{equation}
we can conclude that
\[
u_{h}-u\leq(6K+8K^{2}M_{\delta}+44C_{\beta}K)h^{\gamma}\text{ \  \ in
}[0,1]\times \mathbb{R}.
\]
In the same way, we can derive from Theorem \ref{main theorem} (ii) and
(\ref{ex1}) that%
\[
u-u_{h}\leq(8K+16K^{2}M_{\delta}+88C_{\beta}K)h^{\gamma}\text{ \  \ in
}[0,1]\times \mathbb{R}.
\]
(ii) For $\alpha \in(0,1)$. Combining with Theorem \ref{main theorem} (i),
(\ref{u_mollifer}), and (\ref{ex2}), we know that%
\begin{align*}
u_{h}-u  &  \leq Kh^{^{\frac{\delta}{\alpha}}}+4K\varepsilon+K\varepsilon
^{1-2\alpha/\delta}h+8K^{2}M_{\delta}\varepsilon^{1-\alpha/\delta-\delta
}h^{\delta/\alpha}\\
\text{ \  \ }  &  \text{\  \  \ }+16C_{\beta}Kh^{q_{0}}+8C_{\beta}Ks^{\frac
{1-\alpha}{\alpha}}+C_{\beta}(8\left \Vert \phi \right \Vert _{\infty}+\frac
{4K}{1-\alpha})h^{\alpha q_{0}}.
\end{align*}
Similarly, we can minimize the right-hand side by choosing%
\begin{equation}
\varepsilon=h^{\gamma}\text{\ with \ }\gamma:=\min \left \{  \frac{\delta
}{2\alpha},\frac{\delta^{2}}{\alpha(\alpha+\delta^{2})},\frac{1-\alpha}%
{\alpha},\alpha q_{0}\right \}  , \label{Gama2}%
\end{equation}
we conclude that
\[
u_{h}-u\leq \left(  6K+8K^{2}M_{\delta}+(24+\frac{4}{1-\alpha})C_{\beta
}K+8C_{\beta}\left \Vert \phi \right \Vert _{\infty}\right)  h^{\gamma}\text{
\  \ in }[0,1]\times \mathbb{R}.
\]
In the same way, we can derive from Theorem \ref{main theorem} (ii) and
(\ref{ex1}) that%
\[
u-u_{h}\leq \left(  8K+16K^{2}M_{\delta}+(48+\frac{8}{1-\alpha})C_{\beta
}K+8C_{\beta}\left \Vert \phi \right \Vert _{\infty}\right)  h^{\gamma}\text{
\  \ in }[0,1]\times \mathbb{R}.
\]
(iii) For $\alpha=1$. Together with Theorem \ref{main theorem} (i),
(\ref{u_mollifer}), and (\ref{ex3}), we have%
\begin{align*}
u_{h}-u  &  \leq Kh^{^{\delta}}+4K\varepsilon+K\varepsilon^{1-2/\delta
}h+8K^{2}M_{\delta}\varepsilon^{1-1/\delta-\delta}h^{\delta}\\
\text{ \ }  &  \text{\  \  \ }+6C_{\beta}K\varepsilon^{-1}h+C_{\beta
}K(8\varepsilon^{-1}+4+4q_{0}^{-1})h^{q_{0}}.
\end{align*}
In the same way above, by choosing
\begin{equation}
\varepsilon=h^{\gamma}\text{\ with \ }\gamma:=\min \left \{  \frac{\delta}%
{2},\frac{\delta^{2}}{1+\delta^{2}},\frac{q_{0}}{2}\right \}  ,
\end{equation}
we obtain that
\[
u_{h}-u\leq \left(  6K+8K^{2}M_{\delta}+(18+4q_{0}^{-1})C_{\beta}K\right)
h^{\gamma}\text{ \  \ in }[0,1]\times \mathbb{R}%
\]
and%
\[
u-u_{h}\leq \left(  8K+16K^{2}M_{\delta}+(36+8q_{0}^{-1})C_{\beta}K\right)
h^{\gamma}\text{ \  \ in }[0,1]\times \mathbb{R}\text{.}%
\]
We complete the proof.
\end{proof}

\end{document}